\documentclass[final]{siamltex}
\usepackage{amsmath,latexsym,amssymb,amsfonts}
\usepackage{graphicx,epstopdf,subfigure}
\usepackage[T1]{fontenc}
\usepackage{mathtools}

\newtheorem{thm}{Theorem}[section]
\newtheorem{lem}{Lemma}[section]

\newtheorem{cor}{Corollary}[section]
\newtheorem{rem}{Remark}[section]
\renewcommand{\theequation}{\thesection.\arabic{equation}}

\newcommand{\tmop}[1]{\ensuremath{\operatorname{#1}}}
\newcommand{\capd}{\prescript{C}{0}D_{t}^{\alpha}}

\newcommand{\veps}{\varepsilon}
\newcommand{\bv}{{\mathbf{v}}}
\newcommand{\bff}{{\mathbf{f}}}

\begin{document}

\title{On energy dissipation theory and numerical stability
  for time-fractional phase field equations\thanks{This work is partially
    supported by the NNSF of
    China (under grant numbers 11688101, 11771439, 91530322,
    91630312, 91630203, 11571351, and 11731006), {China
      National Program on Key Basic Research Project
      2015CB856003}, the science challenge project
    (No. TZ2018001), NCMIS, and the youth innovation
    promotion association (CAS).}}

\author{Tao Tang\thanks{	
	Division of Science and Technology, BNU-HKBU United International College,
	Zhuhai, Guangdong, China,
    and Shenzhen International Center for Mathematics, Southern
    University of Science and Technology, Shenzhen, China.
    Email: tangt@sustech.edu.cn.}
  \and Haijun Yu\thanks{NCMIS \& LSEC, Institute of
    Computational Mathematics and Scientific/Engineering
    Computing, Academy of Mathematics and Systems Science,
    Chinese Academy of Sciences, Beijing, 100190 China; {School of
      Mathematical Sciences, University of Chinese Academy
      of Sciences, Beijing, China.} Email:
    hyu@lsec.cc.ac.cn.}
  \and Tao Zhou\thanks{NCMIS \& LSEC, Institute of Computational
    Mathematics and Scientific/Engineering Computing,
    Academy of Mathematics and Systems Science, Chinese
    Academy of Sciences, Beijing, 100190 China. Email:
    tzhou@lsec.cc.ac.cn.}  }

\maketitle

\begin{abstract}
  For the time-fractional phase field models, the corresponding energy
  dissipation law has not been well studied on both the continuous and
  the discrete levels.  In this work, we shall address this open
  issue. More precisely, we prove for the first time that the
  time-fractional phase field models indeed admit an energy dissipation
  law of an integral type. In the discrete level, we propose a class of
  finite difference schemes that can inherit the theoretical energy
  stability. Our discussion covers the time-fractional Allen-Cahn
  equation, the time-fractional Cahn-Hilliard equation, and the
  time-fractional molecular beam epitaxy model. Several numerical
  experiments are carried out to verify the theoretical predictions.  In
  particular, it is observed numerically for both the time-fractional
  Cahn-Hilliard equation and the time-fractional molecular beam epitaxy
  model, there exist a coarsening stage that the energy dissipation rate
  satisfies a power law {scaling} with an asymptotic power $-\alpha/3$,
  where $\alpha$ is the fractional parameter.
\end{abstract}



\section{Introduction}

The phase-field method has been a powerful modeling and simulation tool
in diverse research areas such as material sciences
\cite{AllenCahn_1979, CH, MBE,Li_Liu_2003}, multi-phase flow
\cite{BoschKS_CiCP2018, LowengrubT_1998, Ma_etal_CiCP2017,Anderson_1998,
  qian_molecular_2003, shen_efficient_2015, xu_sharp-interface_2018},
biology and tumor growth \cite{du_phase_2004, Hawkins_tumor_2012,
  Li_tumor_2007,Wise_tumor_2008}, to name a few. Most of the phase field
formulations are based on a free energy function depending on an order
parameter (the phase field) and a diffusive mechanism.  The well-known
examples of phase field models include the Allen-Cahn (AC) equation
\cite{AllenCahn_1979}, the Cahn-Hilliard (CH) equation \cite{CH}, and
the molecular bean epitaxy (MBE) model \cite{MBE,Li_Liu_2003}. A common
feature of the above mentioned phase field models is that their
corresponding free energy admits a dissipation law.

Taking the CH equation as an example, the associated governing equation
yields
\begin{equation}
  \label{eq:Cahn-Hilliard}
  \begin{cases}
    \frac{\partial \phi}{\partial t} + \gamma(- \Delta)
    \left( - \varepsilon \Delta \phi + \frac1\veps F'(\phi)
    \right ) = 0, \quad x \in \Omega \subset \mathbb{R}^d,\ d=2,3,
    \quad 0 < t \leq T, \\
    \phi(x,0)= \phi_0(x), \\
  \end{cases}
\end{equation}
where $\varepsilon$ is an interface width parameter, $\gamma$ is the
mobility, and $F$ is {a double-well potential that is usually} taken the
form
\begin{equation}
  \label{eq:doublewell}
  F(\phi) = \frac{1}{4}(1 -\phi^2)^2.
\end{equation}
For simplicity, we set $\Omega=(0, 2\pi)^d$, and assume that
$\phi(\cdot,t)$ satisfies a periodic boundary condition. The
corresponding free energy functional for the CH equation is defined as
\begin{equation}
  \label{eq:free-energy}
  E(\phi)
  := \int_{\Omega} \Bigl(\frac{\varepsilon}2 |\nabla \phi|^2
  + \frac1\veps F(\phi)\Bigr)  dx.
\end{equation}
The CH equation can be viewed as a gradient flow with the energy
(\ref{eq:free-energy}) in $H^{-1}$.  It is well known that the energy
functional $E$ decreases in time:
\begin{align}
  \label{eq:energy-law_II}
  \frac{d}{dt} E(\phi) = 
  -\int_\Omega \left|\nabla \left( - \varepsilon^2 \Delta \phi
  + F'(\phi) \right)\right|^2 dx \leq 0.
\end{align}
Such an energy dissipation property plays an important role in
developing stable numerical methods for dissipation systems due to its
importance for long time simulations, see e.g.,
\cite{Diegel.etal2017,du_numerical_1991, elliott_global_1993, Eyre_1998,
  Gomez_JCP_2011, Guo.etal2016,
  ShenY_DCDS_2010,wang_efficient_2018,XuT_SINUM_2006, Yan_CiCP_2018,
  YuJZ_CiCP2010} and references therein.

In recent years, fractional-type phase-field models have attracted more
and more attentions \cite{fractional_Mao,fractional_III, fractional_I,
  fractional_Wang_II, fractional_Wang_I, fractional_II,
  fractional_Xu}. For instance, the following fractional type free
energy is investigated in \cite{fractional_Xu}
\begin{equation}
  \label{eq:fractional-free-energy}
  E^\alpha(\phi)
  :=
  \int_{\Omega} \Bigl(\frac{\varepsilon^2}{2} |\nabla^\alpha \phi|^2 + F(\phi)\Bigr) dx,
\end{equation}
where $\nabla^\alpha$ is the fractional gradient
$\nabla^\alpha=(\frac{\partial^\alpha}{\partial x_1},
... ,\frac{\partial^\alpha}{\partial x_d})$ with
$\{\frac{\partial^\alpha}{\partial x_k}\}_k$ being the fractional
derivatives. One is then interested in the following space-fractional CH
equation
\begin{equation}
  \label{eq:Fractional_Cahn_Hilliard_I}
  \frac{\partial \phi}{\partial t} +
  (-\Delta) \left( - \varepsilon^2 \Delta^\alpha \phi + F'(\phi) \right )
  = 0.
\end{equation}
It is obvious that for the modified energy functional
(\ref{eq:fractional-free-energy}), the corresponding energy dissipation
law is
\begin{align}
  \label{eq:energy-law_III}
  \frac{d}{dt} E^\alpha(\phi) \leq 0.
\end{align}
Another interesting approach is to keep the original free energy
(\ref{eq:free-energy}) unchanged, but the associate gradient flow is
considered in $H^{-\alpha}$. This yields the following space-fractional
CH equation \cite{fractional_Mao}
\begin{equation}
  \label{eq:Fractional_Cahn_Hilliard_II}
  \frac{\partial \phi}{\partial t}
  + (- \Delta)^\alpha \left( - \varepsilon^2 \Delta \phi + F'(\phi) \right )
  = 0.
\end{equation}
It is straightforward to verify that its corresponding free energy
admits a dissipation law. As reported in \cite{fractional_Mao}, the
nature of the solution for the fractional CH is qualitatively close to
the behavior of the classical CH equation (\ref{eq:Cahn-Hilliard})
regardless of the size of the parameter $\alpha$.

The time-fractional phase field models have also been investigated
recently. Consider the following time-fractional CH equation
\begin{equation}
  \label{eq:Time-CH}
  \frac{\partial^{\alpha}}{\partial t^{\alpha}} \phi +
  \gamma (- \Delta)\Bigl( - \varepsilon \Delta \phi + \frac1\veps F'(u) \Bigr)
  = 0,
\end{equation}
where $\alpha\in(0,1),$ and $\frac{\partial^{\alpha}}{\partial t^{\alpha}}$
is the Caputo derivative defined as
\begin{equation}
  \frac{\partial^{\alpha}}{\partial t^{\alpha}}\phi
  = \capd\phi(t)
  :=
  \frac{1}{\Gamma(1-\alpha)}\int_{0}^{t}\frac{\phi'(s)}{(t-s)^{\alpha}}ds,
  \quad t>0, \quad\alpha\in(0,1).
  \label{eq:fracCaputo}
\end{equation}
In \cite{fractional_Wang_I}, it is shown numerically that the free
energy admits an energy dissipation law. However, rigorous analysis for
this observed behavior is still open, which is the main motivation of
the present work. Our main contribution is three folds:

\begin{itemize}

\item In the continuous level, we establish the energy dissipation law
  for the time-fractional AC equation, the time-fractional CH equation,
  and the time-fractional MBE model.

\item In the discrete level, we propose a class of finite difference
  schemes satisfying the discrete energy dissipation law for the
  time-fractional problems.

\item We also investigate the coarsening rate of the time-fractional
  phase-field models, and an asymptotic value for a power law is
  obtained.
\end{itemize}

The rest of the paper is organized as following. In Section 2, we shall
establish the energy dissipation law for the time-fractional phase-field
equations. In Section 3, a class of finite difference schemes will be
proposed, whose numerical solutions are shown to satisfy the energy
dissipation property. In Section 4, we shall discuss maximum principle
for the time-fractional AC equation. Numerical examples will be
presented in Section 5 to verify our theoretical results and to predict
an asymptotic power law. We finally give some concluding remarks in
Section 6.

\medskip

\section{Energy dissipation for time-fractional phase field equations}

We first introduce some notations and basic properties for fractional
calculus, see, e.g., \cite{kilbas_theory_2006,podlubny_fractional_1998}.
The Riemann-Liouville fractional integrals for $\alpha\in(0,1)$ on
finite interval $[0,T]$ are defined as
\begin{align*}
  (I_{0+}^{\alpha}f)(t)
  &:=
    \frac{1}{\Gamma(\alpha)}\int_{0}^{t}\frac{f(s)}{(t-s)^{1-\alpha}}ds,
    \quad\mbox{for}\ t\geq0,\\
  (I_{T-}^{\alpha}f)(t)
  &:=
    \frac{1}{\Gamma(\alpha)}\int_{t}^{T}\frac{f(s)}{(s-t)^{1-\alpha}}ds,
    \quad\mbox{for}\ t\leq T.
\end{align*}
The Liouville fractional integrals on the real axis $\mathbb{R}$ are
defined as:
\[
  (I_{+}^{\alpha}f)(t) :=
  \frac{1}{\Gamma(\alpha)}\int_{-\infty}^{t}\frac{f(s)ds}{(t-s)^{1-\alpha}},
  \quad \left(I_{-}^{\alpha}f\right)(t) :=
  \frac{1}{\Gamma(\alpha)}\int_{t}^{\infty}\frac{f(s)ds}{(s-t)^{1-\alpha}}.
\]
For $f\in L^p(0,T), g\in L^q(0,T)$, $p\ge 1, q\ge 1$ and
$\tfrac{1}{p}+\tfrac{1}{q} \le 1+\alpha$ ($p\neq 1$ and $q\neq 1 $ when
$\tfrac{1}{p}+\tfrac{1}{q} = 1+\alpha$), the following integration formulas
holds \cite[Lemma 2.7 on page 76]{kilbas_theory_2006}:
\begin{align}
  &\int_{0}^{T}(I_{0+}^{\alpha}f)(t)g(t)dt
  =
    \int_{0}^{T}f(t)(I_{T-}^{\alpha}g)(t)dt.
    \label{eq:kp2}
\end{align}
Moreover, for $f\in L^p(\mathbb{R}), g\in L^q(\mathbb{R})$, $p, q>1$ and
$\tfrac{1}{p}+\tfrac{1}{q}= 1+\alpha$, we have \cite[(2.3.22) on page
89]{kilbas_theory_2006}
\begin{align}
  &\int_{-\infty}^{\infty}(I_{+}^{\alpha}f)(t)g(t)dt
    =
    \int_{-\infty}^{\infty}f(t)(I_{-}^{\alpha}g)(t)dt.
    \label{eq:LintFracIntByPart}
\end{align}
Furthermore, for $\alpha,\beta>0, \, \alpha+\beta<1/p,$ and
$f\in L^{p}(\mathbb{R}),$ the following semi-group properties for
fractional integrals hold \cite[Lemma 2.19 on page
89]{kilbas_theory_2006}:
\begin{equation}
  (I_{+}^{\alpha}I_{+}^{\beta}f)(t) = (I_{+}^{\alpha+\beta}f)(t),
  \quad
  (I_{-}^{\alpha}I_{-}^{\beta}f)(t) = (I_{-}^{\alpha+\beta}f)(t).
  \label{eq:Lintsemig}
\end{equation}
The Fourier transform for the Liouville factional integrals are given by
\cite[Property 2.15 on page 90]{kilbas_theory_2006}:
\begin{equation}
  (\mathcal{F}I_{+}^{\alpha}f)(\xi) =
  \frac{(\mathcal{F}f)(\xi)}{(-i\xi)^{\alpha}}, \quad
  (\mathcal{F}I_{-}^{\alpha}f)(\xi) =
  \frac{(\mathcal{F}f)(\xi)}{(i\xi)^{\alpha}}, \quad
  \alpha\in(0,1),  \,\, f\in L^{1}(\mathbb{R}),
  \label{eq:LfracFourier}
\end{equation}
where
\[
  (\mp i\xi)^\alpha =|\xi|^{\alpha}e^{\mp i\alpha\pi \mbox{sgn}
    (\xi)/2}.
\]
For $0<\alpha<1$ and $1<p<1/\alpha$, the operator $I_{0+}^\alpha$ and
$I_{T-}^\alpha$ satisfy the following estimates, which is known as the
Hardy-Littlewood theorem \cite[Lemma 2.1.b on page
72]{kilbas_theory_2006}:
\begin{equation}
  \|I_{0+}^\alpha f\|_{L^q(0,T)} \le K \| f\|_{L^p(0,T)},
  \qquad
  \|I_{T-}^\alpha f\|_{L^q(0,T)} \le K \| f\|_{L^p(0,T)},
  \label{eq:H-Ltheorem}
\end{equation}
where $q=p/(1-\alpha p)$, and $K$ is a constant independent of $f$.

Next, we present the following lemma that plays an important role in our
analysis.

\medskip

\begin{lem}
  \label{lem:pdkernel}
  For any given $h\in L^{p}(0,T)$, $p\ge\frac{2}{1+\alpha}$ with
  $\alpha\in(0,1)$, define
  \begin{equation*}
    I_\alpha(h,g) := \frac{1}{\Gamma(\alpha)}\int_{0}^{T}\!\!\int_{0}^{t}\frac{h(s)g(t)}{(t-s)^{1-\alpha}}dsdt.
  \end{equation*}
  Then the following estimates hold:
  \begin{equation}
    I_\alpha(h,h)
    =\int_0^T\! I_{0+}^\alpha h(t) h(t) dt
    \geq \cos\frac{\alpha\pi}{2}\| I_{0+}^{\alpha/2} h \|^2_{L^2(0,T)}
    \geq 0.
    \label{eq:pdkernel}
  \end{equation}
\end{lem}

\smallskip

We briefly outline the proof of the above lemma. It is known that the
kernel $1/t^\alpha$ with $\alpha\in (0,1)$ is positive (see, e.g.,
\cite{mclean_convergence_2009, Mustapha.Schotzau2014,Nohel.Shea1976}),
and this can be verified by using either the Laplace transform or the
Fourier transform. For example, Nohel and Shea \cite{Nohel.Shea1976}
present a proof by checking the kernel function's Laplace transform,
while the property {$I_\alpha(h,h)\ge 0$} is established provided that
$h\in C(0, T)$. Here we can extend this result to the space of
$h\in L^p(0,T)$ with $p\ge\frac{2}{1+\alpha}$.  In this case,
\eqref{eq:pdkernel} can be shown {by} using the Fourier transform
technique. To show this, one can perform the zero extension for $h$ from
$[0,T]$ to $(-\infty,\infty)$, apply the semi-group property
\eqref{eq:Lintsemig}, and use the integration formula
\eqref{eq:LintFracIntByPart} to rewrite the integrand of the outer
integration as $(I_{+}^{\alpha/2}h)(I_{-}^{\alpha/2}h)$.  Hence, the
desired result (\ref{eq:pdkernel}) follows by combining the Fourier
transform \eqref{eq:LfracFourier}, the convolution theorem and the
Parseval's theorem with Hardy-Littlewood inequality
\eqref{eq:H-Ltheorem}.

\medskip

Notice that $h$ in Lemma \ref{lem:pdkernel} is assumed to take values in
$\mathbb{R}$. Nevertheless, when $h$ takes values in some Hilbert space,
the result can be extended by using the orthonormal bases expansion
argument, see, e.g., \cite{Mustapha.Schotzau2014}.

A direct extension of Lemma \ref{lem:pdkernel} leads to the following
corollary.

\smallskip

\begin{cor}
  \label{cor:pdkernel}
  For any given $h,g\in L^{p}(0,T)$ with $p\ge\frac{2}{2-\alpha}$, define
  \begin{equation}\label{eq:Aalpha}
    A_\alpha(h,g) := I_{1-\alpha}(h,g) =
    \frac{1}{\Gamma(1-\alpha)}
    \int_{0}^{T}\int_{0}^{t}\frac{h(s)g(t)}{(t-s)^{\alpha}}ds dt.
  \end{equation}
  Then the following estimates hold:
  \begin{align}
	\label{eq:pdkernelRiesz}
	A_\alpha(h,h)
    &= \frac{1}{2}\frac{1}{\Gamma(1-\alpha)}
      \int_{0}^{T}\int_{0}^{T}\frac{h(s)h(t)}{|t-s|^{\alpha}}ds dt \nonumber \\
    &\geq \sin\frac{\alpha\pi}{2}\| I_{0+}^{(1-\alpha)/2} h \|^2_{L^2(0,T)}
      \geq 0.
  \end{align}
\end{cor}

\subsection{The time-fractional Allen-Cahn equation}

Consider the following time-fractional AC equation:
\begin{equation}
  \label{eq:fracAC}
  \frac{\partial^{\alpha}}{\partial t^{\alpha}}\phi
  = \gamma\bigl(\veps\Delta\phi-\frac{1}{\veps}F'(\phi)\bigr),
  \quad \alpha \in (0,1),
  \quad (x,t) \in \Omega \times [0,T]
\end{equation}
with the homogeneous boundary condition
\begin{equation}
  \label{eq:fracAC_BC}
  \phi(x,t) = 0,
  \quad
  (x,t) \in \partial \Omega \times[0,T],
\end{equation}
where $\varepsilon$ is the thickness of the phase interface, $\gamma$ is
a mobility constant. Here we assume that $F$ admits the following
property: $F\in C^2(R)$ and there exist two constants $M_1< 0 < M_2$
such that
\begin{equation}
  \label{eq:pot_coercive}
  F'(M_1) = F'(M_2) =0; \,\,  F'(u) > 0,\ \forall\, u>M_2\;\; \mbox{and}\ F'(u)<0, \; \forall \, u<M_1.
\end{equation}
\begin{rem}
  The quartic growth double-well potential \eqref{eq:doublewell}
  satisfies this property with $M_1=-1$ and $M_2=1$.
\end{rem}

Notice that the fractional order $\alpha\in (0,1)$ in \eqref{eq:fracAC},
and when $\alpha=1$, \eqref{eq:fracAC}-\eqref{eq:fracAC_BC} is the
standard AC equation which satisfies a well-known energy dissipation
property
\begin{equation}
  \label{eq:ACedissDiffType}
  \frac{d}{dt}E[\phi]
  = -\frac{1}{\gamma} \Bigl{\|} \frac{\partial}{\partial t}  \phi\Bigr{\|} ^{2},
\end{equation}
or
\begin{equation}
  \label{eq:ACedissIntType}
  E[\phi(T)] - E[\phi(0)]
  = -\frac{1}{\gamma}\int_0^T \Bigl{\|} \frac{\partial}{\partial t}  \phi\Bigr{\|} ^{2} dt,
\end{equation}
where $E[\phi]$ is the system energy:
\begin{equation}\label{eq:E_ACcontinuous}
  E\left[\phi\right]
  = \frac{\veps}{2} \|\nabla\phi\|^{2}
  + \frac1\veps\bigl\langle F(\phi),1 \bigr\rangle.
\end{equation}
Here we use $\left\langle \cdot,\cdot\right\rangle $ to denote the
$L^{2}(\Omega)$ inner product in the spatial domain; and we denote by
$\left\Vert \cdot\right\Vert $ the standard $L^{2}(\Omega)$ norm.

We first address the question whether the time-fractional AC equation
{still} satisfies an energy dissipation law similar to
(\ref{eq:ACedissDiffType}) or \eqref{eq:ACedissIntType}.

\medskip

\begin{thm}
  \label{thm:ACEdiss}
  Consider Eq. \eqref{eq:fracAC} with the homogeneous boundary condition
  \eqref{eq:fracAC_BC} (or homogeneous Neumann/periodic boundary
  condition), if the initial energy $E[\phi(0)]$ is finite, then the
  following energy holds:
  \begin{equation}
    \label{eq:FACEenergylaw}
    E\left[\phi(T)\right]-E\left[\phi(0)\right]
    = -\frac{1}{\gamma}\int_{\Omega}A_{\alpha}(\phi_{t},\phi_{t})dx\leq0.
  \end{equation}
\end{thm}

\begin{proof}
  Multiplying both sides of (\ref{eq:fracAC}) by -$\phi_{t}$ and taking
  integration on the resulting equation yield
  \[
    -\int_{0}^{T}\int_{\Omega}\frac{\partial^{\alpha}\phi}{\partial
      t^{\alpha}}\phi_{t}dxdt
    =
    \gamma\int_{0}^{T}\int_{\Omega}(-\varepsilon\Delta\phi+
    \frac{1}{\varepsilon}F'(\phi))\phi_{t}dxdt
    =
    \gamma\int_{0}^{T}\frac{d}{dt}E[\phi]dt.
  \]
  Consequently,
  \begin{equation}
    \label{eq:ACEdiss0}
    E\left[\phi(T)\right]-E\left[\phi(0)\right]
    = -\frac{1}{\gamma \Gamma(1-\alpha)}
    \int_{\Omega}\int_{0}^{T}\int_{0}^{t}
    \frac{\phi_{t}(x,s)}{(t-s)^{\alpha}}ds\phi_{t}(x,t)dtdx.
  \end{equation}
  The combining (\ref{eq:ACEdiss0}) and Corollary \ref{cor:pdkernel} yields
  the desired property \eqref{eq:FACEenergylaw}.
\end{proof}
\medskip

\begin{rem} {\rm Notice that Eq. \eqref{eq:FACEenergylaw} is
    an energy dissipation law of integral type, by which
    we have
    \[
      E[\phi(T)]\le E[\phi(0)].
    \]
    Moreover, the dissipated energy
    $\frac{1}{\gamma}\int_{\Omega}A_{\alpha}(\phi_{t},\phi_{t})dx$ is
    also bounded by $E[\phi(0)]$, thus the solution is energy stable and
    $\| I_{0+}^{(1-\alpha)/2} \phi_t \|^2_{L^2(0,T)}$ is bounded.
    However, the conclusion \eqref{eq:FACEenergylaw}, in general, does
    not lead to $\frac{d}{dt} E\le 0$ or
    $\frac{d^\alpha}{dt^\alpha} E\le 0.$ }
\end{rem}

\subsection{The time-fractional Cahn-Hilliard equation}

The analysis of the time-fractional AC equation can be extended to the
time-fractional CH equation.

\smallskip
\begin{thm}
  \label{thm:CHEdiss} Consider the potential function $F$ described by
  \eqref{eq:pot_coercive} and total energy $E$ defined by
  \eqref{eq:E_ACcontinuous}. The time-fractional CH equation
  \begin{equation}
    \frac{\partial^{\alpha}}{\partial t^{\alpha}}\phi
    = \gamma\Delta\mu,\quad \mu
    =  -\veps\Delta\phi+\frac1\veps F'(\phi)
    \label{eq:fracCH}
  \end{equation}
  with periodic boundary conditions or no-flux boundary conditions
  \begin{equation}
    \frac{\partial\mu}{\partial n}\Big|_{\partial\Omega} = 0,
    \quad\frac{\partial\phi}{\partial n}\Big|_{\partial\Omega} = 0
    \label{eq:FCHEbc}
  \end{equation}
  satisfies the energy dissipation law
  \begin{equation}
    E\left[\phi(T)\right]-E\left[\phi(0)\right]
    =
    -\frac{1}{\gamma}\int_{\Omega}A_{\alpha}(\nabla\psi,\nabla\psi)dx\leq0,
    \label{eq:CHEdiss}
  \end{equation}
  where $\psi=-\Delta^{-1}\phi_{t}$ is the solution of the following
  equation
  \begin{equation}
    -\Delta\psi=\phi_{t}
    \label{eq:Laplace4CHE}
  \end{equation}
  with periodic or homogeneous Neumann boundary condition
  $\partial_{n}\psi\big|_{\partial\Omega}=0$.
\end{thm}
\smallskip

\begin{proof}
  We first show that the time-fractional CH equation conserves the total
  mass. More precisely, if $\phi$ is the solution of (\ref{eq:fracCH})
  with the periodic boundary condition or the no-flux boundary condition
  (\ref{eq:FCHEbc}), then
  \begin{equation}
    \int_{\Omega}\phi(x,t)dx =
    \int_{\Omega}\phi(x,0)dx,
    \qquad
    \forall \; t\geq 0.
    \label{eq:FCHEmassConserv}
  \end{equation}
  To see this, by integrating both sides of the first equation in
  (\ref{eq:fracCH}) in the physical domain one obtain
  \begin{align*}
    0
    &= \int_{\Omega}\frac{\partial^{\alpha}\phi}{\partial
      t^{\alpha}}dx
      = \frac{1}{\Gamma(1-\alpha)}
      \int_{\Omega}\int_{0}^{t}\frac{1}{(t-s)^{\alpha}}\frac{\partial\phi(x,s)}{\partial t}dsdx\\
    & =
      \frac{1}{\Gamma(1-\alpha)}\int_{0}^{t}\frac{1}{(t-s)^{\alpha}}
      \Bigl(\int_{\Omega}\frac{\partial}{\partial t}\phi(x,s)dx\Bigr)ds,
      \quad\forall\; t\geq 0.
  \end{align*}
  Then we obtain
  \[
    \frac{d}{dt}\int_{\Omega}\phi(x,t)dx =
    \int_{\Omega}\frac{\partial}{\partial t}\phi(x,t) = 0,
    \quad \mbox{for a.e.}\ t\geq0.
  \]
  Then (\ref{eq:FCHEmassConserv}) follows by integrating the above
  equation. It follows from (\ref{eq:FCHEmassConserv}) that
  $\int_{\Omega}\phi_{t}dx=0$. Thus, $\psi$ in (\ref{eq:Laplace4CHE}) is
  well defined. Pairing the first equation of \eqref{eq:fracCH} with
  $-\frac{1}{\gamma}\psi$, the second equation with $\phi_{t}$, and
  summing up the two resulting equations, we get
  \begin{equation} \label{eexx1}
    -\frac{1}{\gamma}\left\langle
      I_{0+}^{1-\alpha}\nabla\psi,\nabla\psi\right\rangle =
    \frac{d}{dt}E.
  \end{equation}
  Consequently, the energy dissipation law (\ref{eq:CHEdiss}) follows by
  integrating (\ref{eexx1}) from $0$ to $T$ and by using Lemma
  \ref{lem:pdkernel}.
\end{proof}

\subsection{The time-fractional MBE Model}

Now we consider the time-fractional MBE model:
\begin{equation}
  \frac{\partial^{\alpha}}{\partial t^{\alpha}}\phi
  = \gamma\big(-\veps\Delta^{2}\phi+\frac1\veps\nabla\cdot \bff_m(\nabla\phi)\big),
  \quad \alpha\in(0,1),
  \label{eq:TFMBE}
\end{equation}
where ${\bff}_m({\bv})= \partial F_m(\bv)/{\partial \bv}$ and
$F_m({\bv})$ is defined as (see, e.g., \cite{Li_Liu_2003})
\begin{equation}\label{eexx2}
  F_{m}({\bv})=
  \begin{cases}
    \frac{1}{4}\big(|{\bv}|^{2}-1\big)^{2}, \quad & \mbox{for model with slope selection},\\
    -\frac{1}{2}\ln\big|1+|{\bv}|^{2}\big|, & \mbox{for model without slope selection}.
  \end{cases}
\end{equation}
It can be verified that
${\bff}_m({\nabla\phi})=(|\nabla\phi|^{2}-1)\nabla\phi$ in the model
with slop selection, and
${\bff}_m({\nabla\phi})=-\frac{\nabla\phi}{1+|\nabla\phi|^{2}}$ in the
model without slop selection.

Using similar arguments as in the last subsection, we can show that the
time-fractional MBE model has the following energy dissipation property.

\smallskip
\begin{thm}
  \label{thm:TFMBEEnergyDiss}
  If the time-fractional MBE model satisfies periodic boundary condition
  or no-flux boundary condition
  \begin{equation}
    \partial_{n}\Delta\phi|_{\partial\Omega}=\partial_{n}\phi|_{\partial\Omega}=0,
    \label{eq:TFMBEbc}
  \end{equation}
  then the solution of (\ref{eq:TFMBE}) satisfies the energy dissipation
  law
  \begin{equation}
    E_{m}[\phi(T)]-E_{m}[\phi(0)]\leq-\frac{1}{\gamma}\int_{\Omega}A_{\alpha}(\phi,\phi)dx\leq0,
    \label{eq:TFCHEdisEnergyDiss-1-1}
  \end{equation}
  where
  \begin{equation} \label{eexx4}
    E_{m}(\phi) =
    \frac{\veps}2\|\Delta\phi\|^{2}+\frac1\veps\bigl\langle
    F_m(\nabla\phi),1\bigr\rangle.
  \end{equation}
\end{thm}

\section{Energy stable finite difference schemes}

In this section, we shall design energy stable finite difference schemes
for the time-fractional phase field models. To this end, we shall first
review a commonly-used finite difference scheme.

Let us first consider the following time-fractional diffusion equation
\begin{equation*}
  \frac{\partial^{\alpha}u}{\partial t^{\alpha}}=\Delta u+g(x,t),
  \quad\alpha\in(0,1),
\end{equation*}
which can be viewed as a linearized version of the time-fractional AC
equation (\ref{eq:fracAC}). For ease of notation, we consider the
one-dimensional case, i.e.,
$\Delta u=\frac{\partial^{2}u}{\partial x^{2}}$, with $x$ being the spatial
variable. Let $\tau$ be the time step size, $t_{k}=k\tau$, $u^{k}(\cdot)$
be the numerical approximation of $u(\cdot,t_{k})$.  By applying the
classical L1 scheme (see, e.g., \cite{liao_discrete_2019, lin_finite_2007})
to the time-fractional derivative and treating other terms in an implicit
way, one gets the following scheme:

\begin{equation}
  \sum_{j=0}^{k}b_{j}\frac{u^{k+1-j}(x)-u^{k-j}(x)}{\tau} =
  \frac{\partial^{2}u^{k+1}(x)}{\partial x^{2}}+g(x,t_{k+1}),
  \label{eq:fdDiff}
\end{equation}
where
\begin{equation}
  b_{j}
  = \frac{1}{\Gamma(1-\alpha)}\int_{j\pi}^{(j+1)\pi}\frac{1}{t^{\alpha}}dt
  = \frac{\tau^{1-\alpha}}{\Gamma(2-\alpha)}\left[(j+1)^{1-\alpha}-j^{1-\alpha}\right],
  \quad j\geq0.
  \label{eq:coefb}
\end{equation}

The derivation of the left hand side of (\ref{eq:fdDiff}) is given as
below
\begin{align*}
  &\frac{\partial^{\alpha}u}{\partial t^{\alpha}}(x,t_{k+1})
    =
    \frac{1}{\Gamma(1-\alpha)}\int_{0}^{t_{k+1}}\frac{u_{t}(x,s)}{(t_{k+1}-s)^{\alpha}}ds\\
  &= 
    \sum_{j=0}^{k}\frac{u(x,t_{j+1})-u(x,t_{j})}{\tau}\frac{1}{\Gamma(1-\alpha)}
    \int_{t_{j}}^{t_{j+1}}\frac{ds}{(t_{k+1}-s)^{\alpha}}+r_{\tau}^{k+1}\\
  &=
    \sum_{j=0}^{k}\frac{u(x,t_{j+1})-u(x,t_{j})}{\tau}
    \frac{\tau^{1-\alpha}}{\Gamma(2-\alpha)}
    [(k+1-j)^{1-\alpha}-(k-j)^{1-\alpha}]+r_{\tau}^{k+1}\\
  &=
    \sum_{j=0}^{k}b_{k-j}\frac{u(x,t_{j+1})-u(x,t_{j})}{\tau}+r_{\tau}^{k+1},
\end{align*}
where the integer derivative $\frac{\partial}{\partial t}u(x,t)$ in time
interval $[t_j, t_{j+1}]$ is approximated with a first order Euler
scheme \cite{sun_fully_2006}. The finite difference scheme for the
fractional differential operator is obtained by dropping the remainder
$r_{\tau}^{k+1}$.

Suppose the spatial discretization using Galerkin approach is accurate
enough. Then the $H^{1}$ stability is available as the L1 discretization
of the fractional derivative satisfies the special property:
\begin{equation}
  \sum_{j=0}^{k}b_{j}\frac{u^{k+1-j}(x)-u^{k-j}(x)}{\tau}
  =
  \frac{1}{\tau}\Bigl[b_{0}u^{k+1}-\sum_{j=0}^{k-1}(b_{j}-b_{j+1})u^{k-j}-b_{k}u^{0}\Bigr],
  \label{eq:bterm_equiv}
\end{equation}
and
\begin{equation}\label{eq:bprop}
  b_{k}>0,
  \quad  b_{k}-b_{k+1}>0,
  \quad  \sum_{j=0}^{k-1}(b_{j}-b_{j+1})_{}+b_{k}=b_{0},
  \quad  \forall\ k\geq 0.
\end{equation}
By this property, if one pair the scheme \eqref{eq:fdDiff} with
$u^{k+1}$, then all the cross terms $u^{j}u^{k+1}$ $(j=0,\ldots, k)$ can
be bounded by $\frac12 [(u^{j})^{2}+(u^{k+1})^{2}]$. Hence the $H^1$
stability can be proved by a simple mathematical induction
\cite{lin_finite_2007}.

Before providing rigorous nonlinear stability analysis, let us make an
assumption on the bulk potential function $F(\phi):$
$F(\phi) \in C^2 (\mathbb{R})$, and there exists a finite constant $L$
such that
\begin{equation}
  \label{eq:pot_Lip}
  \max_{u\in R}|F''(u)| \le L.
\end{equation}
\begin{rem}
  There are many ways to modify the potential $F(\phi)$ such that
  (\ref{eq:pot_Lip}) is satisfied. One possible way is to to lower the
  far-ends nonlinearity(\cite{caffarelli_l_1995, ShenY_DCDS_2010,
    condette_spectral_2011, wang_energy_2018}). E.g.  we consider the
  following double-well potential with quadratic growth:
    \begin{equation}
      \label{eq:doublewell_tr}
      {F}(\phi)=
      \begin{cases}
        \frac{11}{2}(\phi-2)^{2}+6(\phi-2)+\frac94, &\phi>2, \\
        \frac{1}{4 }(\phi^{2}-1)^{2}, &\phi\in [-2,2], \\
        \frac{11}{2}(\phi+2)^{2} - 6(\phi+2)+\frac94,
        &\phi<-2.
      \end{cases}
    \end{equation}
    It can be verified that the above potential satisfy
    (\ref{eq:pot_Lip}).
  \end{rem}

\medskip

\subsection{The time-fractional AC equation}

We now consider the time-fractional AC equation (\ref{eq:fracAC}). We
adopt the L1 scheme for the linear part of (\ref{eq:fracAC}) and use a
stabilization technique for the nonlinear bulk force.  This leads to the
following semi-discretized scheme for (\ref{eq:fracAC}):
\begin{align}
  &\frac{1}{\gamma}\sum_{j=0}^{k}b_{j}\frac{\phi^{k+1-j}(x)-\phi^{k-j}(x)}{\tau}
    \nonumber \\
  & \quad\quad ={} \veps\Delta\phi^{k+1}-\frac{1}{\veps}f(\phi^{k})
    -\frac{S}{\gamma}(\phi^{k+1}-\phi^{k}),
    \quad  k\geq0,
    \label{eq:ACFDstab}
\end{align}
where $f(\phi)=F'(u)$, $S$ is a sufficiently large positive constant,
$\tau=T/n$ is the time step size, and $\{b_{j}\}$ are defined by
(\ref{eq:coefb}).

To show the energy stability of the above numerical scheme, we first
present the following lemma.

\medskip
\begin{lem}
  \label{lem:dispd}
  For any
  $(u_{1},\ldots,u_{n})^{T}\in {\mathbb{R}}^{n}$, define
  \[
    B := 2\sum_{k=1}^{n}\sum_{j=1}^{k}b_{|k-j|}u_{j}u_{k}.
  \]
  Then we have
  \begin{align}
    B
    &= \sum_{k=1}^{n}b_{0}u_{k}^{2} +
      \sum_{k=1}^{n}\sum_{j=1}^{n}b_{|k-j|}u_{j}u_{k} \geq
      \sum_{k=1}^{n} {b_0} u_{k}^{2},
      \label{eq:lem31a} \\
     \label{eq:lem31b}
    B
    &\ge  \frac{2}{\tau}  \sin\frac{\alpha\pi}{2}\| I_{0+}^{(1-\alpha)/2} u^n(t)
      \|^2_{L^2(0,T)}
      + s_n\sum_{k=1}^{n}u_{k}^{2},
  \end{align}
  where
  $s_{n}=\bigl(\frac{n+1}2\bigr)^{-\alpha}\frac{1}{\Gamma(1-\alpha)}\tau^{1-\alpha}>0$
  and the piecewise constant function $u^n(t)$ is defined by
  \begin{equation}
    u^{n}(t)=
    \begin{cases}
      u_{\left\lfloor t/\tau\right\rfloor +1}, & 0\leq t<T,\\
      0, & \mbox{otherwise}.
    \end{cases}\label{eq:ugridfunc}
  \end{equation}
  In (\ref{eq:ugridfunc}), $\left\lfloor t\right\rfloor $ stands for the
  integer part of real number $t$.
\end{lem}
\medskip

\begin{proof}
  Note that Eq. \eqref{eq:lem31a} is well known (see, e.g.,
  \cite{Le.etal2016,Tang1993b}). We only need to prove the inequality
  \eqref{eq:lem31b}. We shall prove it by converting $B$ into the form
  of Corollary \ref{cor:pdkernel}.  First, convert
  $\left\{ u_{j},j=1,\ldots,n\ \right\} $ into a piecewise constant
  function $u^n(t)$ on $[0,T]$ as {in \eqref{eq:ugridfunc}.}  Obviously,
  $u^{n}(t)\in L^{2}(0,T)$.  Then by Lemma \ref{lem:pdkernel} or
  Corollary \ref{cor:pdkernel}, we have
  \begin{align*}
    0
    &\leq
      \frac{2}{\tau}A_{\alpha}(u^{n},u^{n})
      = \frac{1}{\tau}\frac{1}{\Gamma(1-\alpha)}
      \int_{0}^{T}\int_{0}^{T}  \frac{u^{n}(s)u^{n}(t)}{|t-s|^{\alpha}}dsdt\\
    &= 
      \frac{1}{\tau}\frac{1}{\Gamma(1-\alpha)}
      \sum_{k=1}^{n}\int_{(k-1)\tau}^{k\tau}
      u_{k}\int_{0}^{T}\frac{u^{n}(s)}{|t-s|^{\alpha}}dsdt\\
    &= 
      \frac{1}{\tau}\frac{1}{\Gamma(1-\alpha)}
      \sum_{k=1}^{n}u_{k}\sum_{j=1}^{n}u_{j}
      \int_{(k-1)\tau}^{k\tau}\int_{(j-1)\tau}^{j\tau}\frac{1}{|t-s|^{\alpha}}dsdt
      =
      \sum_{k=1}^{n}\sum_{j=1}^{n}u_{j}u_{k}\tilde{b}_{|k-j|},
  \end{align*}
  where
  \begin{align*}
    \tilde{b}_{|k|}
    &=
      \frac{1}{\Gamma(1-\alpha)}\frac{1}{\tau}
      \int_{k\tau}^{(k+1)\tau}\int_{0}^{\tau}\frac{1}{|t-s|^{\alpha}}dsdt\\
    &=
      \frac{\tau^{1-\alpha}}{\Gamma(3-\alpha)}
      \left((k+1)^{2-\alpha}-2k^{2-\alpha}+(k-1)^{2-\alpha}\right),
      \quad  k\geq1,\\
    \tilde{b}_{0}
    &=
      \frac{2}{\Gamma(2-\alpha)}\frac{1}{\tau}
      \int_{0}^{\tau}t^{1-\alpha}dt=\frac{2}{\Gamma(3-\alpha)}\tau^{1-\alpha}.
  \end{align*}
  It is easy to see that $b_{|k|}$ is an approximation of
  $\tilde{b}_{|k|}$ by evaluating the integration using a one side
  quadrature rule. To prove $B$ is positive definite, we need to prove
  that the difference between $B$ and $2A_\alpha/\tau$ in the
  off-diagonal parts can be controlled by the difference in the diagonal
  part. To show this, for the diagonal term we have
  \[
    2b_{0}-\tilde{b}_{0} =
    \frac{2\tau^{1-\alpha}}{\Gamma(2-\alpha)} -
    \frac{2\tau^{1-\alpha}}{\Gamma(3-\alpha)} =
    2b_{0}(1-\frac{1}{2-\alpha})=\frac{2}{2-\alpha}\frac{\tau^{1-\alpha}}{\Gamma(1-\alpha)}
    \ge 0.
  \]
  For the off-diagonal term, let
  $G(x):=\frac{1}{2-\alpha}(x+1)^{2-\alpha}-\frac{1}{2-\alpha}x{}^{2-\alpha}$,
  then we have
  \begin{align*}
    & \tilde{b}_{|k|}-b_{|k|}\\
    & =
      \frac{\tau^{1-\alpha}}{\Gamma(2-\alpha)}
      \Bigl(\frac{(k+1)^{2-\alpha}-2k^{2-\alpha}+(k-1)^{2-\alpha}}{2-\alpha}
      -\left[(k+1)^{1-\alpha}-k^{1-\alpha}\right]\Bigr)\\
    & =
      \frac{\tau^{1-\alpha}}{\Gamma(2-\alpha)}\left(G(k)-G(k-1)-G'(k)\right)\\
    & =
      \frac{\tau^{1-\alpha}}{\Gamma(2-\alpha)}
      \frac{1}{2}(1-\alpha)\left(x^{-\alpha}-(x+1)^{-\alpha}\right),
      \quad\mbox{for some } x\in[k-1,k]\\
    &\ge 0, \qquad \forall\ k\geq 1.
  \end{align*}
  Thus we obtain,
  \begin{align*}
    \sum_{k=1}^{m}\left|\tilde{b}_{|k|}-b_{|k|}\right|
    &=
      \frac{\tau^{1-\alpha}}{\Gamma(2\!-\!\alpha)}
      \Bigl(\frac{({m}+1)^{2-\alpha}\!-{m}^{2-\alpha}\!-1}{2-\alpha}
      -\left[({m}+1)^{1-\alpha}\!-1\right]\Bigr)\\
    &\leq
      \frac{\tau^{1-\alpha}}{\Gamma(1\!-\!\alpha)}
      \Bigl(\frac{1}{2\!-\!\alpha} - \frac12({m}+1)^{-\alpha}\Bigr),
      \quad {1\le m \le n-1}.
  \end{align*}
  A direct calculation shows that the column sum of the off-diagonals
  are bounded by
  \begin{equation*}
  	c_0:=\frac{\tau^{1-\alpha}}{\Gamma(1\!-\!\alpha)}
  	\Bigl(\frac{2}{2\!-\!\alpha} - \bigl(\frac{n+1}2\bigr)^{-\alpha}\Bigr).
  \end{equation*}
  Hence $C=\left\{ c_{k-j}\right\} _{k,j=1}^{n}$ with
  $c_{k}=b_{|k|}-\tilde{b}_{|k|}$ for $k=\pm1,\ldots,\pm n$ is a
  symmetric positive definite $M$-matrix.  We then have
  \begin{align*}
    B
    &=
      \frac{2}{\tau}A_{\alpha}(u^{n},u^{n})
      + \sum_{k=1}^{n}\left[2b_{0}-\tilde{b}_{0}-c_{0}\right]
      u_{k}^{2}+\sum_{k=1}^{n}\sum_{j=1}^{n}c_{k-j}u_{k}u_{j}\\
    &\geq
      \frac{2}{\tau}A_{\alpha}(u^{n},u^{n})
      + \sum_{k=1}^{n}\left[2b_{0}-\tilde{b}_{0}-c_{0}\right]u_{k}^{2}\\
    &=
      \frac{2}{\tau}A_{\alpha}(u^{n},u^{n})
      + s_n\sum_{k=1}^{n}u_{k}^{2}.
  \end{align*}
  The proof is complete.
\end{proof}
\medskip

We are now ready to give the following result indicating that the
proposed numerical scheme is energy stable.
\smallskip

\begin{thm}
  \label{thm:TFACEdisEnergyDiss}
  The numerical solution of (\ref{eq:ACFDstab}) with a modified bulk
  potential function (\ref{eq:doublewell_tr}) satisfies the following
  discrete energy dissipation law
  \begin{align}
    E[\phi^{n}]-E[\phi^{0}]
    &\leq 
      -\frac{{b_{0}}}{2\gamma\tau}\sum_{k=0}^{n-1}\|\delta_{t}\phi^{k+1}\|^{2}
      \nonumber\\
    &\qquad{}
      -\sum_{k=0}^{n-1}\left\{ \frac\veps2\|\nabla\delta_{t}\phi^{k+1}\|^{2}
      +\Bigl\langle \frac{S}{\gamma}-\frac{1}{2\veps}f'(\xi^{k}),
      (\delta_{t}\phi^{k+1})^{2}\Bigr\rangle
      \right\},
    \label{eq:TFACEdisEnergyDiss}
  \end{align}
  where $\delta_{t}\phi^{k+1}:=\phi^{k+1}-\phi^{k}$, providing that
  \begin{equation} \label{eq:fracAC_ESC}
  	S+ \frac{b_0}{2\tau} \geq \frac{\gamma L}{2\veps},
  \end{equation}
  where $L$ is given by (\ref{eq:pot_Lip}).  If $S\geq \gamma L/2\veps$,
  then the scheme is unconditional energy stable in the sense that
  \[
    E[\phi^{n}]\le E[\phi^{0}], \quad \forall\;\tau>0,\ n> 0.
  \]
\end{thm}

\begin{proof}
  Multiplying both sides of (\ref{eq:ACFDstab}) by
  $\delta_{t}\phi^{k+1}$, and integrating in space, the resulting
  right-hand side is given by
  \begin{align*}
    \tmop{RHS}
    &=
      -\frac\veps2\|\nabla\phi^{k+1}\|^{2}
      +\frac\veps2\|\nabla\phi^{k}\|^{2}
      -\frac\veps2\|\nabla\delta_{t}\phi^{k+1}\|^{2}
      -\Bigl\langle \frac1\veps f(\phi^{k})\delta_{t}\phi^{k+1}
      +\frac{S}{\gamma}(\delta_{t}\phi^{k+1})^{2},1\Bigr\rangle \\
    &=
      -\frac\veps2\|\nabla\phi^{k+1}\|^{2}
      +\frac\veps2\|\nabla\phi^{k}\|^{2}
      -\frac\veps2\|\nabla\delta_{t}\phi^{k+1}\|^{2} \\
    &\qquad { }
      -\Bigl\langle \frac1\veps F(\phi^{k+1})-\frac1\veps F(\phi^{k}) +
      \Bigl(\frac{S}{\gamma} -\frac{1}{2\veps}f'(\xi^{k})\Bigr)
      (\delta_{t}\phi^{k+1})^{2}, 1\Bigr\rangle,
  \end{align*}
  where $\xi^{k}(x)$ is between $\phi^{k}(x)$ and $\phi^{k+1}(x)$. On
  the other hand, the resulting left-hand side is given by
  \[
    \tmop{LHS} = \frac{1}{\gamma\tau}
    \int_{\Omega}\sum_{j=0}^{k}b_{j}\delta_{t}\phi^{k+1-j}\delta_{t}\phi^{k+1}dx.
  \]
  Summing up both sides for $k=0,\ldots,n-1$, we get
  \begin{align*}
  	E[\phi^{n}]-E[\phi^{0}] 
  	&=
     -\frac{1}{\gamma\tau}
     \int_{\Omega}\sum_{k=0}^{n-1}\sum_{j=0}^{k}b_{j}
     \delta_{t}\phi^{k+1-j}\delta_{t}\phi^{k+1} dx\\
    &\qquad {}
      -\sum_{k=0}^{n-1}\left\{ \frac\veps2 \|\nabla\delta_{t}\phi^{k+1}\|^{2}
      +\Bigl\langle
      \frac{S}{\gamma}-\frac{1}{2\veps}f'(\xi^{k}), (\delta_{t}\phi^{k+1})^{2}
      \Bigr\rangle  \right\} .
  \end{align*}
  The desired energy estimate (\ref{eq:TFACEdisEnergyDiss}) follows by
  using Lemma \ref{lem:dispd}.
\end{proof}

\medskip

\subsection{The time-fractional CH equation}

The scheme (\ref{eq:ACFDstab}) can be easily extended to the
time-fractional CH equation (\ref{eq:fracCH}) with the double-well
potential function of quadratic growth (\ref{eq:doublewell_tr}):
\begin{align}
  &
    \frac{1}{\gamma}\sum_{j=0}^{k}b_{j}\frac{\phi^{k+1-j}(x)-\phi^{k-j}(x)}{\tau}=\Delta\mu^{k+1},
    \label{eq:TFCHEBDF1}\\
  &
    \mu^{n+1}=-\veps\Delta\phi^{k+1}+\frac{1}{\veps}f(\phi^{k})+\frac{S}{\gamma}(\phi^{k+1}-\phi^{k}).
    \label{eq:TFCHEBDF1_2}
\end{align}

Similar to the time-fractional AC equation case, one can prove the
following energy dissipation property.

\smallskip

\begin{thm}
  \label{thm:TFCHEdisEnergyStab}
  Consider the numerical scheme
  (\ref{eq:TFCHEBDF1})-(\ref{eq:TFCHEBDF1_2}) with a modified bulk
  potential function (\ref{eq:doublewell_tr}). Then the numerical
  solution of (\ref{eq:TFCHEBDF1})-(\ref{eq:TFCHEBDF1_2}) satisfies the
  following discrete energy dissipation property:
  \begin{align}
    E[\phi^{n}]-E[\phi^{0}]\nonumber
	& \leq 
   -\frac{{b_{0}}}{2\gamma\tau}
   \sum_{k=0}^{n-1}\|\delta_{t}\phi^{k+1}\|_{H^{-1}}^{2}\\
    & \qquad {}
      -\sum_{k=0}^{n-1} \left\{
      \frac{\veps}{2}\|\nabla\delta_{t}\phi^{k+1}\|^{2}
      + \Bigl\langle
      \frac{S}{\gamma}-\frac{1}{2\veps}f'(\xi^{k}),(\delta_{t}\phi^{k+1})^{2}
      \Bigr\rangle
      \right\},
    \label{eq:TFCHEdisEnergyDiss}
  \end{align}
  providing that
  $\sqrt{\frac{b_0\varepsilon}{\gamma\tau}} +\frac{S}{\gamma} \ge
  \frac{L}{2\varepsilon}$. If $S\geq\frac{\gamma L}{2\veps}$, then the
  scheme is unconditional energy stable in the sense that
  \[
    E[\phi^{n}]\le E[\phi^{0}], \quad \forall\; {\tau>0,}\ n> 0.
  \]
\end{thm}


\subsection{The time-fractional MBE equation}
The scheme (\ref{eq:ACFDstab}) and the relevant analysis also apply to
the time-fractional MBE model (\ref{eq:TFMBE}) without slope selection,
in which we have
\begin{equation} \label{eexx8}
  \bff_m'(\bv) := \frac{\partial
    \bff_m(\bv)}{\partial \bv}=\frac{2\bv^2 -
    (|\bv|^2+1)I}{(1+|\bv|^2)^2},
\end{equation}
where $I$ is an identity matrix. The corresponding numerical scheme is
\begin{align}
  & \frac{1}{\gamma}\sum_{j=0}^{k}b_{j}\frac{\phi^{k+1-j}(x)-\phi^{k-j}(x)}{\tau}\nonumber\\
  &\qquad\quad = -\veps\Delta^{2}\phi^{k+1} + \frac{1}{\veps}\nabla\cdot \bff_m(\nabla\phi^{k})
    + \frac{S}{\gamma}(\Delta\phi^{k+1}-\Delta\phi^{k}),
    \quad k\geq0.
   \label{eq:TFMBEBDF1}
\end{align}
Similarly, one can prove the following energy dissipation property.

\medskip

\begin{thm}
  \label{thm:TFMBEdisEnergyStab}
  {Consider} the time-fractional MBE model (\ref{eq:TFMBE}) with
  periodic boundary condition or no-flux boundary condition
  (\ref{eq:TFMBEbc}).  Then the numerical scheme (\ref{eq:TFMBEBDF1})
  satisfies the following discrete energy law:
  \begin{align}
    \nonumber
    E_{m}[\phi^{n}]-E_{m}[\phi^{0}]
    &\leq
    - \frac{{b_{0}}}{2\gamma\tau}
      \sum_{k=0}^{n-1}\|\delta_{t}\phi^{k+1}\|^{2}
      -\sum_{k=0}^{n-1}
      \frac{\veps}{2}\|\nabla^{2}\delta_{t}\phi^{k+1}\|^{2} \\
    &\qquad\qquad\qquad
      -\sum_{k=0}^{n-1}
      \Bigl\langle  \frac{S}{\gamma} - \frac{1}{2\veps} \bff_m'(\xi^{k}),
      (\nabla\delta_{t}\phi^{k+1})^{2} \Bigr\rangle,
      \label{eq:TFCHEdisEnergyDiss-1}
  \end{align}
  providing that
  $\sqrt{\frac{b_0\varepsilon}{\gamma\tau}} +\frac{S}{\gamma} \ge
  \frac{1}{2\varepsilon}\lambda_\text{max}\big(\bff_m'(\xi^{k})\big)$. Here
  $\lambda_\text{max}\big(\bff_m'(\xi^{k})\big)$ means the largest
  eigenvalue of the matrix $\bff_m'(\xi^{k})$. In particular, if we
  choose $S\geq \frac{\gamma}{16\veps}$, then the numerical scheme
  (\ref{eq:TFMBEBDF1}) for the time-fractional MBE model without slope
  selection is unconditionally energy stable for any time step size.
\end{thm}
  \begin{proof}
    We remark that the inequality (\ref{eq:TFCHEdisEnergyDiss-1}) can be
    obtained by the standard energy method similar to those employed in
    the previous subsections. To make the last term in
    (\ref{eq:TFCHEdisEnergyDiss-1}) non-negative, one only needs that
    \[
      \frac{S}{\gamma} - \frac{1}{2\veps}\lambda_\text{max}\big(\bff_m'(\xi^{k})\big) \ge 0.
    \]
    A direct calculation using (\ref{eexx8}) shows that
    {$\lambda_\text{max}\big(\bff_m'(\xi^{k})\big) \le 1/8
      $}. Consequently, the above inequality holds provided that
    $S\geq \frac{\gamma}{16\veps}$.
  \end{proof}

\medskip

\begin{rem}
  One can also resort to the convex-splitting approach
  \cite{elliott_global_1993, Eyre_1998} to design energy stable
  schemes. Take the classical double well potential
  \eqref{eq:doublewell} as an example, let
	\begin{equation} \label{eq:fsplit} f_{i}(\phi) = \phi^3,\quad
      f_{e}(\phi) = \phi.
	\end{equation}
	For the time-fractional CH equation, the corresponding convex
    splitting scheme reads
	\begin{align}
      &
        \frac{1}{\gamma}\sum_{j=0}^{k}b_{j}\frac{\phi^{k+1-j}(x)-\phi^{k-j}(x)}{\tau}=\Delta\mu^{k+1},
        \label{eq:TFCHEBDF1cs}\\
      &
        \mu^{n+1}=-\veps\Delta\phi^{k+1}+\frac{1}{\veps}f_i(\phi^{k+1}) -\frac{1}{\veps}f_e(\phi^{k}).
        \label{eq:TFCHEBDF1_2cs}
	\end{align}
	For the time-fractional MBE model with slope selection, the
    corresponding scheme yields
	\begin{align}
      & \frac{1}{\gamma}\sum_{j=0}^{k}b_{j}\frac{\phi^{k+1-j}(x)-\phi^{k-j}(x)}{\tau}
        \nonumber\\
      &\qquad = -\veps\Delta^{2}\phi^{k+1}
        +\frac{1}{\veps}\nabla\cdot f_i(\nabla\phi^{k+1})
        -\frac{1}{\veps}\nabla\cdot f_e(\nabla\phi^{k}),
        \quad k\geq0.
           \label{eq:TFMBEBDF1cs}
	\end{align}
	Here $f_i(\bv)=|\bv|^2 \bv$. These schemes can be proved to be
    unconditionally stable but at each time step one needs to solve a
    nonlinear system.
\end{rem}

\smallskip
\section{Maximum principle for the Allen-Cahn equation}

Similar to the classical Allen-Cahn equation (see, e.g.,
\cite{Shen.etal2016c}), we can establish a discrete maximum principle
for the time fractional AC equation. Note that if the discrete maximum
principle is valid, then the global Lipschitz condition on $f$ can be
removed as the numerical solutions are bounded by the initial data. In
other words, the modified bulk potential function
(\ref{eq:doublewell_tr}) may not be needed; instead the standard
double-well potential (\ref{eq:doublewell}) can be used.

Next, we show that scheme \eqref{eq:ACFDstab} with the standard
double-well potential (\ref{eq:doublewell}) satisfies a discrete maximum
principle provided that the time step size is sufficiently small.

\medskip

\begin{thm} \label{thm:MPdiscrete} Suppose $\phi_0 \in C^0$ and
  $M_1 \le \phi_0(x) \le M_2$ for all $x\in\bar\Omega$.  Let
  $\{ \phi^k, k \ge 1 \}$ be the solution of the semi-discretized scheme
  \eqref{eq:ACFDstab} with a standard double-well potential that
  satisfies property \eqref{eq:pot_coercive}. If
  \begin{equation}
    \label{eq:mpdisc_cond}
	\frac{b_0-b_1}{\tau} + {S} \ge
    \gamma\frac{\max_{M_1\le u\le
        M_2}|f'(u)|}{\varepsilon},
  \end{equation}
  then we have
  \begin{equation}
    M_1 \le \phi^k(x) \le M_2,\quad \forall\ k\ge 1, \, {x\in\bar\Omega}.
  \end{equation}
\end{thm}
\begin{proof}
  Using \eqref{eq:bterm_equiv} to rewrite the scheme \eqref{eq:ACFDstab}
  as
  \begin{align}
	&\frac{b_0}{\gamma\tau} \phi^{k+1}
   + \frac{S}{\gamma} \phi^{k+1}
   - \varepsilon \Delta \phi^{k+1} \nonumber\\
	& \qquad =
   \frac{1}{\gamma\tau}
   \Bigl[\sum_{j=0}^{k-1}(b_{j}-b_{j+1})\phi^{k-j}+b_{k}\phi^{0}\Bigr]
   + \frac{S}{\gamma}\phi^{k}
   -\frac{1}{\veps}f(\phi^{k}),
   \quad  k\geq0,
   \label{eq:ACFDstab2}
  \end{align}
  We first consider the case $k=0.$ By using the fact that
  $f(\phi^k)=f(\phi^k)-f(M_2)=f'(\xi)(\phi^k-M_2)$ we obtain
  \begin{align*}
    & \frac{b_0}{\gamma\tau} \phi^{1}
      + \frac{S}{\gamma} \phi^{1}
      - \varepsilon \Delta \phi^{1} \\
    & =
      \frac{b_0}{\gamma\tau} \phi^0
      + \frac{S}{\gamma}\phi^0
      + \frac{1}{\varepsilon} f'(\xi)(M_2-\phi^0) 	\\
    & = \frac{b_0}{\gamma\tau} M_2
      + \frac{S}{\gamma} M_2
      +  \Bigl(\frac{1}{\varepsilon} f'(\xi)-\frac{S}{\gamma}
      - \frac{b_0}{\gamma\tau}\Bigr) (M_2-\phi^0) 	\\
    & \le \frac{b_0}{\gamma\tau} M_2 + \frac{S}{\gamma} M_2.
    \end{align*}
	The last inequality is a result of condition
    \eqref{eq:mpdisc_cond}. Now, let $x_0$ be a maximum point of
    $\phi^{1}$, if $\phi^{1}(x_0)>M_2$, we get
    \[
      \frac{b_0}{\gamma\tau} \phi^{1}
      + \frac{S}{\gamma} \phi^{1}
      - \varepsilon \Delta \phi^{1}
      >  \frac{b_0}{\gamma\tau}M_2 + \frac{S}{\gamma} M_2
      \quad \text{at}\;\;  x=x_0,
    \]
    which contradicts with the above inequality. Thus, we have
    $\max_{x\in\bar\Omega}\phi^{1}(x)\le M_2$.

    Next, suppose $\max_{x\in\bar\Omega}\phi^j(x)\le M_2$ hold for
    $j=0,\ldots, k$. Then by using \eqref{eq:bprop}, the fact
    $f(\phi^k)=f(\phi^k)-f(M_2)=f'(\xi)(\phi^k-M_2)$ and the condition
    \eqref{eq:mpdisc_cond}, we get
  \begin{align*}
	& \frac{b_0}{\gamma\tau} \phi^{k+1}
   + \frac{S}{\gamma} \phi^{k+1}
   - \varepsilon \Delta \phi^{k+1} \\
	& \leq
   \frac{b_1}{\gamma\tau} M_2
   + \frac{b_0-b_1}{\gamma\tau}\phi^k
   + \frac{S}{\gamma}\phi^k + \frac{1}{\varepsilon} f'(\xi)(M_2-\phi^k) 	\\
	& = \frac{b_0}{\gamma\tau} M_2
   + \frac{S}{\gamma} M_2
   +  \Bigl(\frac{1}{\varepsilon} f'(\xi)-\frac{S}{\gamma}
   - \frac{b_0-b_1}{\gamma\tau}\Bigr) (M_2-\phi^k) 	\\
	& \le \frac{b_0}{\gamma\tau} M_2 + \frac{S}{\gamma} M_2.
  \end{align*}
  Again, let $x_0$ be a maximum point of $\phi^{k+1}$, if
  $\phi^{k+1}(x_0)>M_2$, we obtain
  \[
    \frac{b_0}{\gamma\tau} \phi^{k+1}
	+ \frac{S}{\gamma} \phi^{k+1}
	- \varepsilon \Delta \phi^{k+1}
    >  \frac{b_0}{\gamma\tau}M_2 + \frac{S}{\gamma} M_2
	\quad\text{at}\ x=x_0,
  \]
  which contradicts the above inequality.  Thus, we have
  $\max_{x\in\bar\Omega}\phi^{k+1}(x)\le M_2$. By mathematical induction,
  we have $\max_{x\in\bar\Omega}\phi^j(x)\le M_2$ for any $j\ge 1$. The
  lower bound can be proved similarly.
\end{proof}

\medskip

It is well-known that the solution of linear time-fractional parabolic
equations admits some kind of initial singularity (see, e.g.,
\cite{Allen.etal2016,Brunner2004, Jin.etal2013,Mustapha.Schotzau2014,
  stynes_error_2017,Tang1993b}):
\begin{equation}
  \phi(t) \in C([0,T]),\quad |\phi_t| \le C_\alpha t^{\alpha-1}
  \label{eq:Solu_regularity}.
\end{equation}
The well-posedness and the limited regularity of the time-fractional AC
equation \eqref{eq:fracAC}-\eqref{eq:fracAC_BC} with $F$ satisfying
\eqref{eq:pot_Lip} is studied recently by Jin et
al. \cite{jin_numerical_2018} in the setting of nonlinear sub-diffusion
equation. It is proved in \cite{jin_numerical_2018} that if
$\phi_0\in H_0^1(\Omega)\cap H^2(\Omega)$, then
\eqref{eq:fracAC}-\eqref{eq:fracAC_BC} admits a unique solution $\phi$
satisfying
\begin{align}
  \label{eq:tfrac_AC_reg1}
  & \phi \in C^\alpha([0,T];
    L^2(\Omega)) \cap C([0,T]; H_0^1(\Omega)\cap H^2(\Omega)),
    \quad \partial_t^\alpha \phi \in C([0,T]; L^2(\Omega)), \\
  & \partial_t \phi(t) \in
    L^2(\Omega)\quad
    \text{and}\quad \| \partial_t
    \phi(t) \|_{L^2(\Omega)} \le c
    t^{\alpha-1},\quad \text{for}\
    t\in(0,T].
    \label{eq:tfrac_AC_reg2}
  \end{align}
  Correspondingly, the numerical solutions of scheme \eqref{eq:ACFDstab}
  with a modified bulk potential function (\ref{eq:doublewell_tr}) can
  be proved to satisfy the following convergence property using standard
  technique in \cite{jin_numerical_2018}
  \begin{equation}
	\label{eq:tfracAC_stab_conv}
	\max_{1\le n\le N} \| \phi(n\tau) - \phi^n \|_{L^2(\Omega)} \le c\tau^\alpha.
  \end{equation}
  Since the piecewise extension
  $\phi_\tau(x,t):= (n+1-t/\tau)\phi^n + (t/\tau-n) \phi^{n+1}$
  ($n\tau \le t \le n\tau + \tau$) converges to the weak solution
  $\phi(x,t)$ in $C([0,T]; L^2(\Omega)),$ using the fact that \
  $\phi_\tau$ is bounded uniformly (Theorem \ref{thm:MPdiscrete}) for
  sufficiently small $\tau$ and $\phi$ is a continuous function yields
  the following maximum principle result.
  
  \medskip
  \begin{thm}
    \label{eq:MPweak}
    Let $\phi(x,t)$ is a weak solution of
    \eqref{eq:fracAC}-\eqref{eq:fracAC_BC} with a potential function
    $F(\phi)$ described by \eqref{eq:pot_coercive}. Suppose that
    $\phi(x,0)=\phi_0 \in H_0^1(\Omega)\cap H^2(\Omega)$,
    $\Omega\subset \mathbb{R}^d, d=2,3$ and $M_1 \le \phi_0(x) \le M_2$
    for all $x\in\bar\Omega$, then
	\begin{equation*}
      M_1 \le \phi(x,t) \le M_2,
      \quad \forall\ a.e.\ {(x,t)\in \bar\Omega\times[0,T]}.
	\end{equation*}
  \end{thm}

  \medskip
  \begin{rem} 
	\label{thm:tFACEmaxp} The above theorem provides a maximum principle
    for the week solution. On the other hand, one can also obtain a
    maximum principle directly for the strong solution.  Suppose the
    initial value $\phi_0$ and the spatial domain $\Omega$ are
    sufficiently smooth such that the solution $\phi$ of
    \eqref{eq:fracAC}-\eqref{eq:fracAC_BC} with a potential function
    $F(\phi)$ described by \eqref{eq:pot_coercive} satisfies
	\begin{equation}\label{eq:assum}
      \Delta\phi \in C^0((0,T]\times\Omega), \qquad
      \phi_t\in C^0((0,T]\times\Omega),
      \quad
      \phi\in C^0([0,T]\times\Omega).
	\end{equation}
    If the initial data is bounded, i.e., $ M_1 \le \phi_0 \le M_2$ for
    all $x\in \Omega$, then one can use a standard technique to prove
    that
    \[
      M_1 \le \phi(x,t) \le M_2   \quad \forall x\in \Omega, \; \;t>0.
    \]
    We point out that the regularity assumptions (\ref{eq:assum}) have
    been shown reasonable recently, as reported in \cite{Du}.
  \end{rem}

\medskip

\section{Numerical experiments}

In this section, numerical schemes proposed in the last section will be
employed to study the coarsening rates of the time-fractional
phase-field models.  We solve the phase-field equations in
$\Omega=[0,L_x]\times[0,L_y]$ with periodic boundary conditions.  A
Fourier-Galerkin method is used for the spatial discretization. To
enhance the computational efficiency, the fast sum-of-exponential
algorithm developed in \cite{jiang_fast_2017} is used to evaluate the
history part of the time-fractional derivatives.

\subsection{Numerical results for the time-fractional AC
  equation}

\begin{figure}[htbp]
  \centering
  \includegraphics[width=0.3\textwidth]{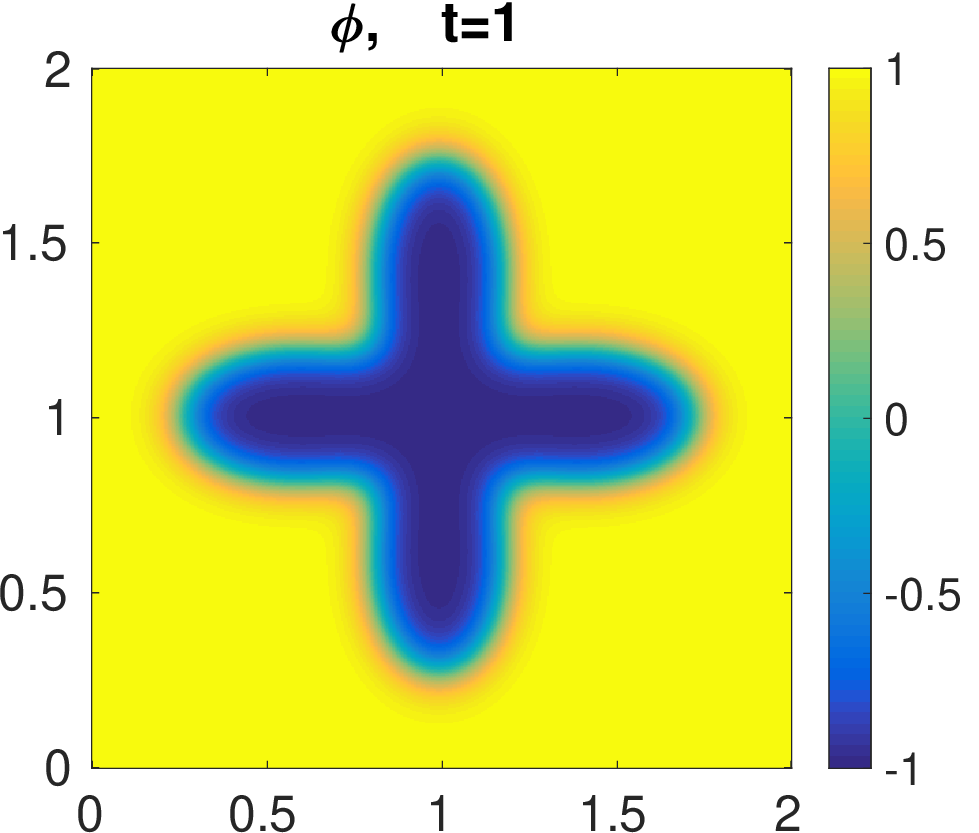}
  \includegraphics[width=0.3\textwidth]{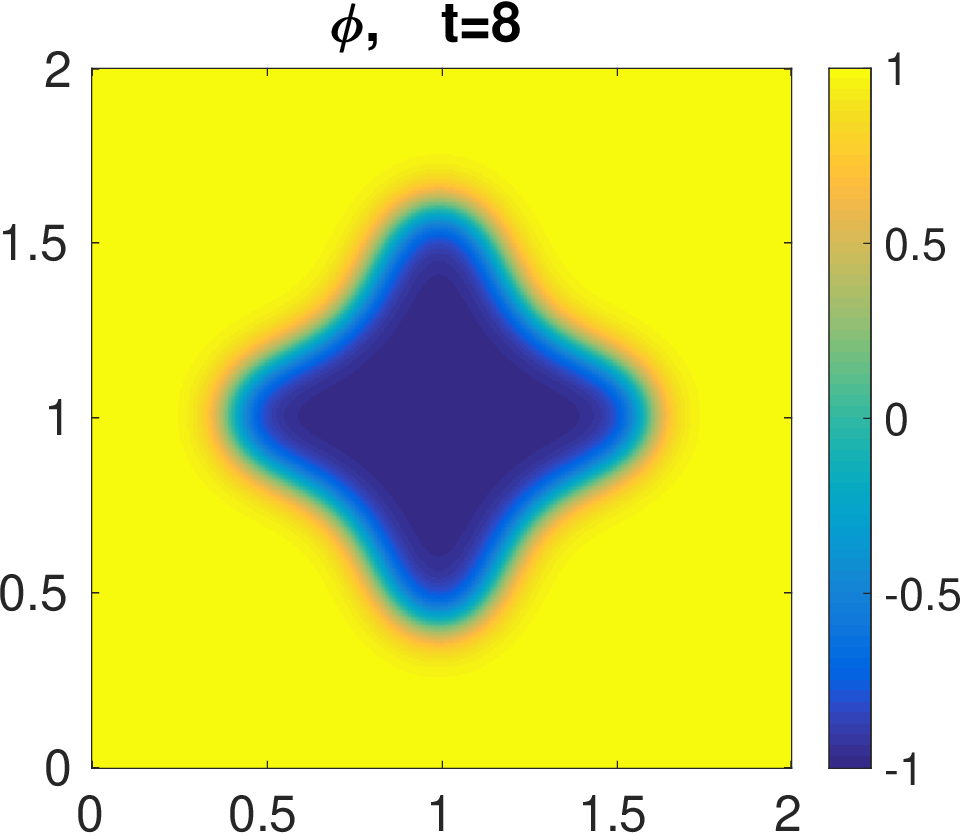}
  \includegraphics[width=0.3\textwidth]{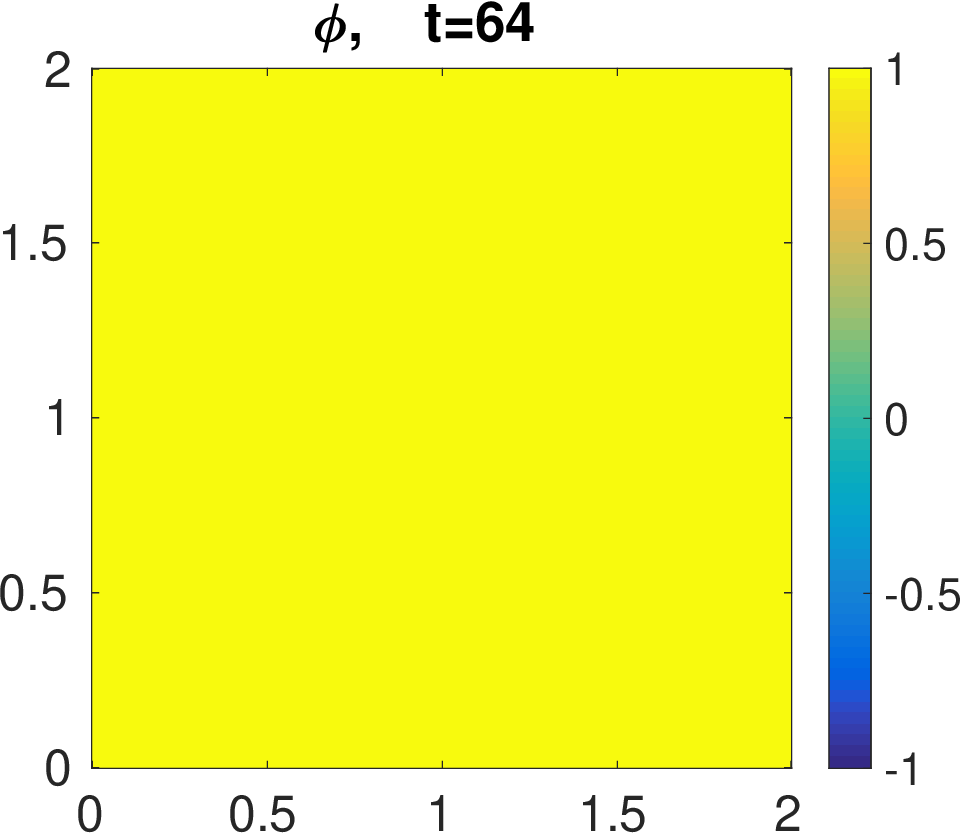}\\
  \includegraphics[width=0.3\textwidth]{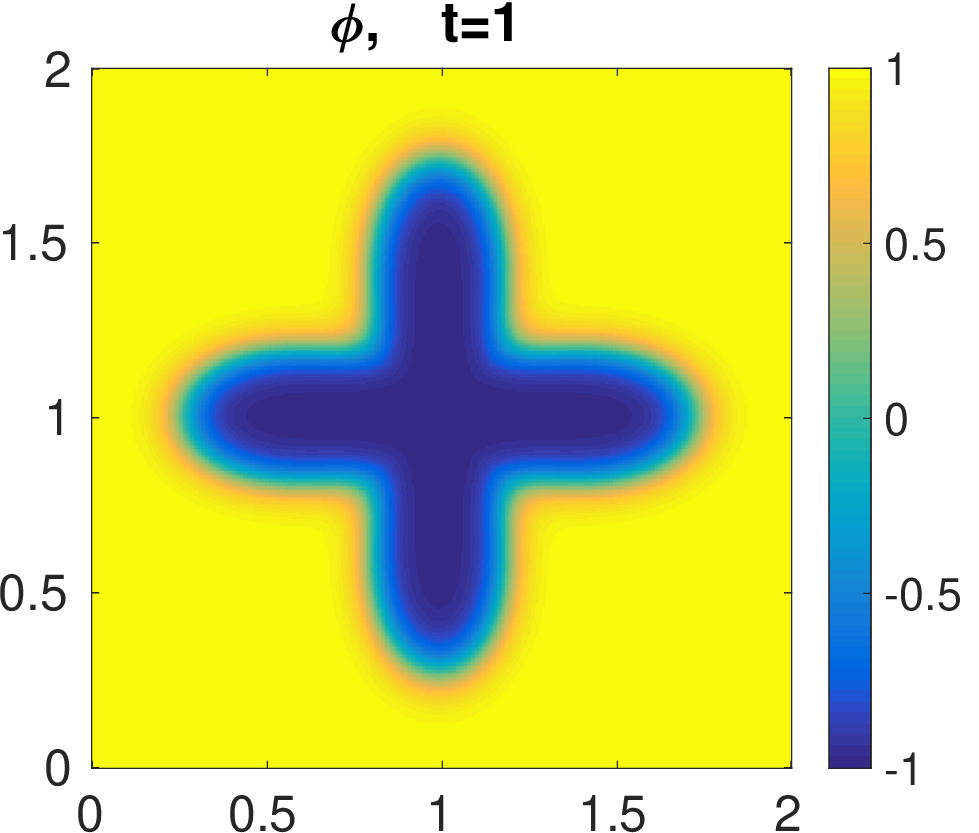}
  \includegraphics[width=0.3\textwidth]{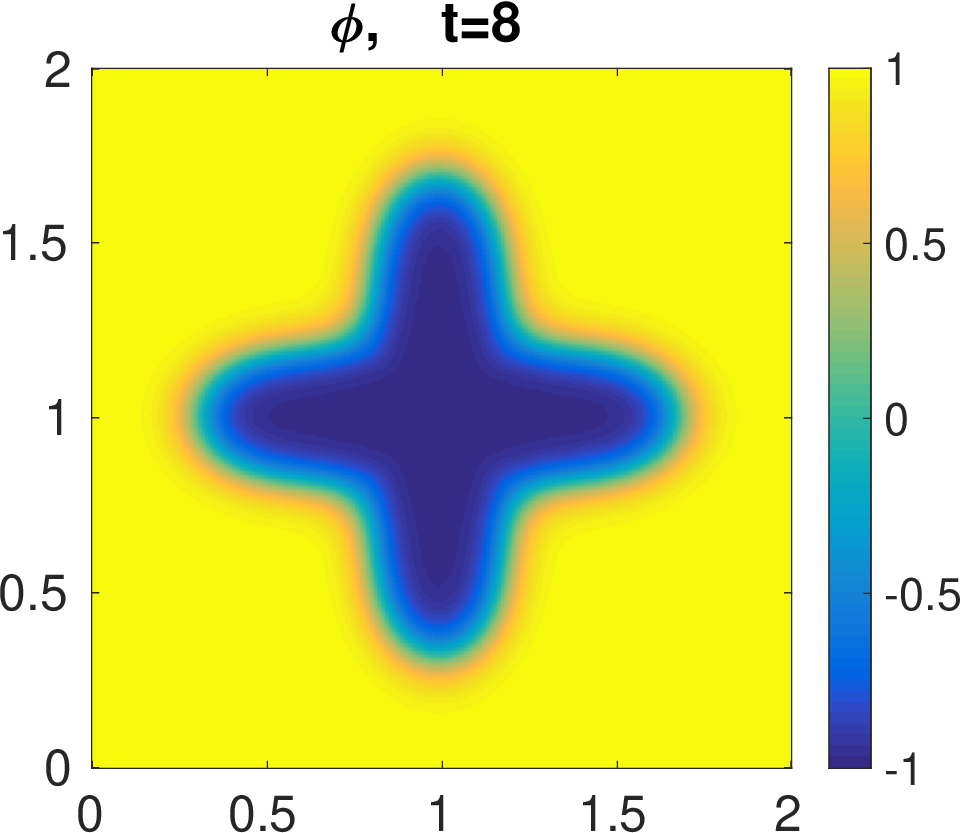}
  \includegraphics[width=0.3\textwidth]{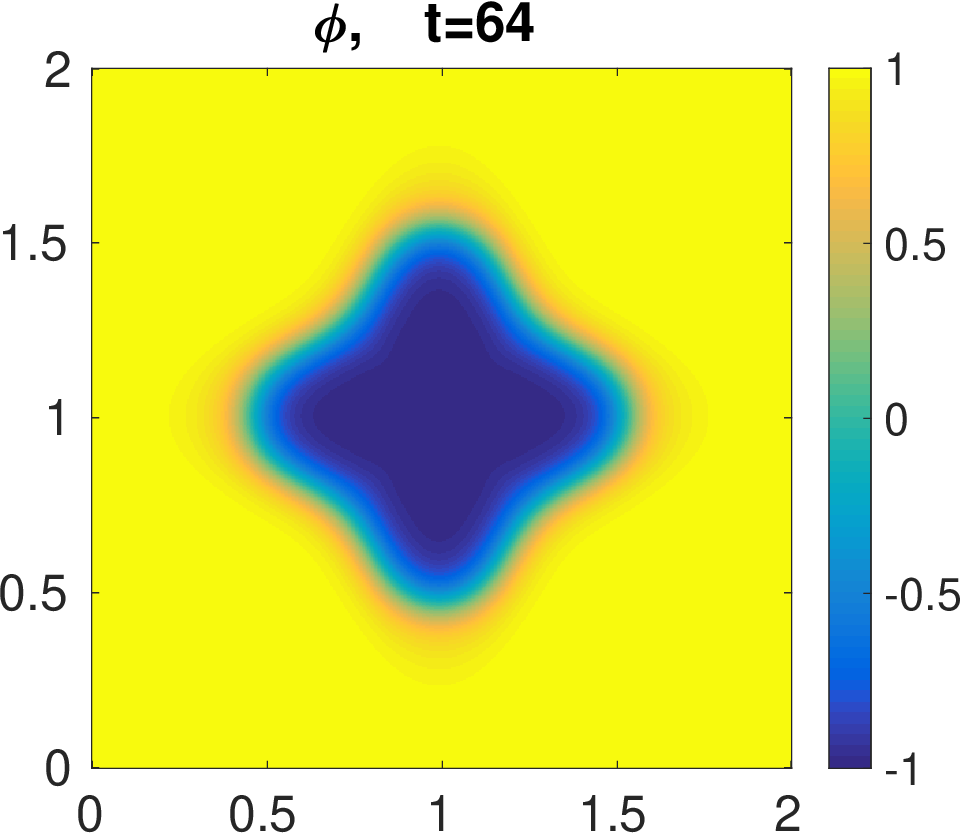}\\
  \includegraphics[width=0.3\textwidth]{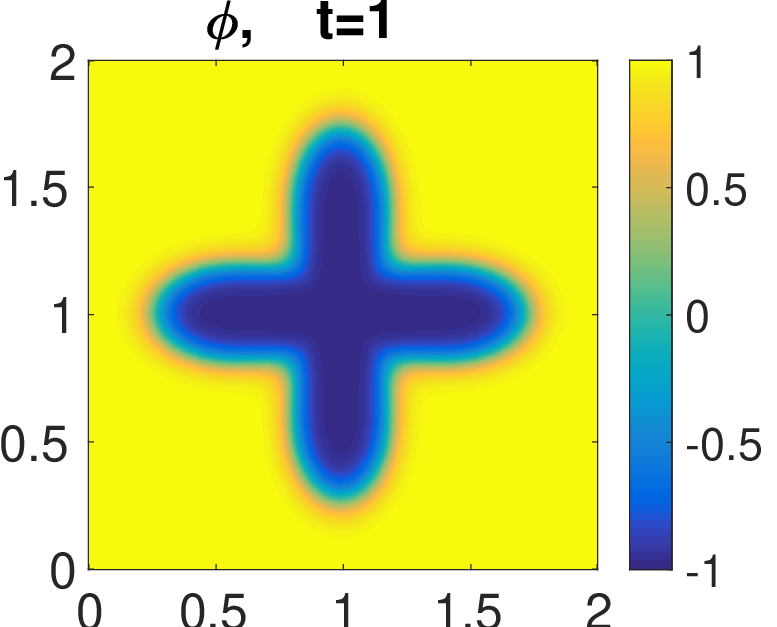}
  \includegraphics[width=0.3\textwidth]{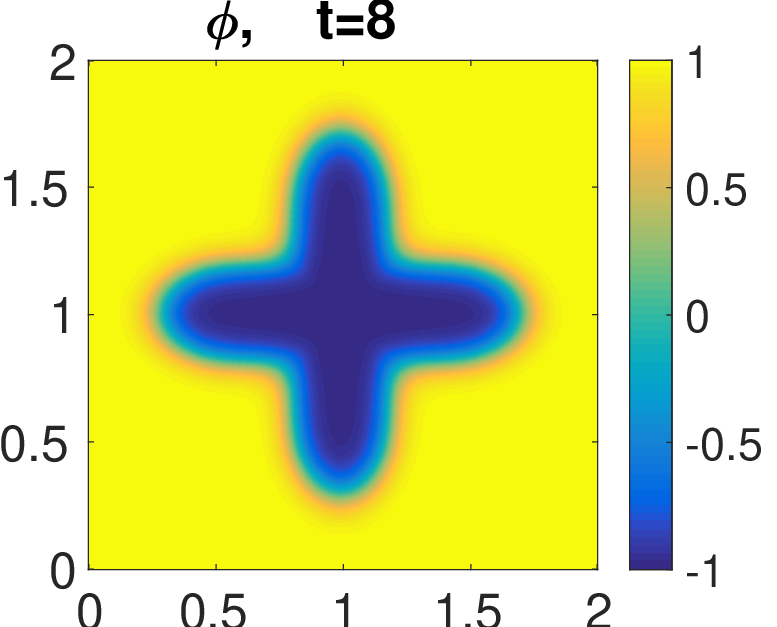}
  \includegraphics[width=0.3\textwidth]{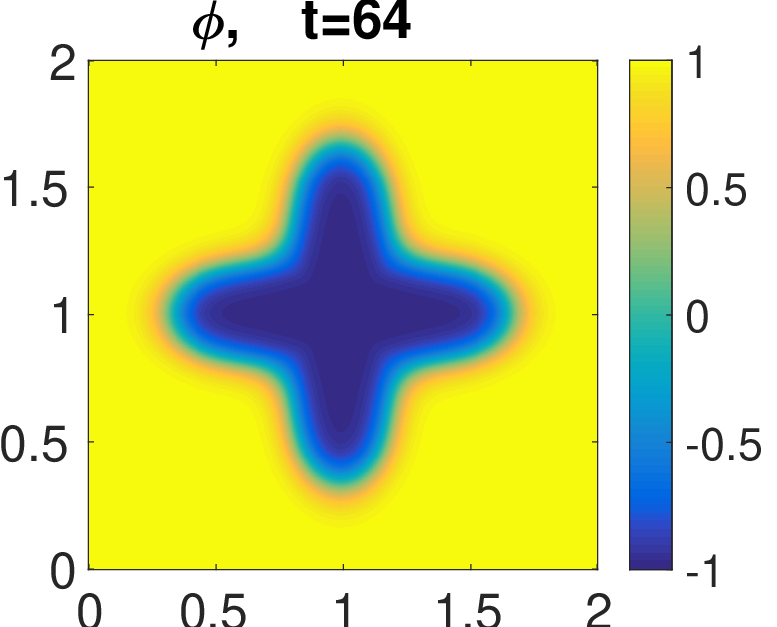}
  \caption{The snapshots of the solution of the
    time-fractional AC equation with $\alpha=1, 0.5, 0.3$
    (top,middle, bottom row, respectively). }
  \label{fig:1FACEsolu}
\end{figure}
\begin{figure}[htpb!]
  \centering
  \includegraphics[width=0.55\textwidth]{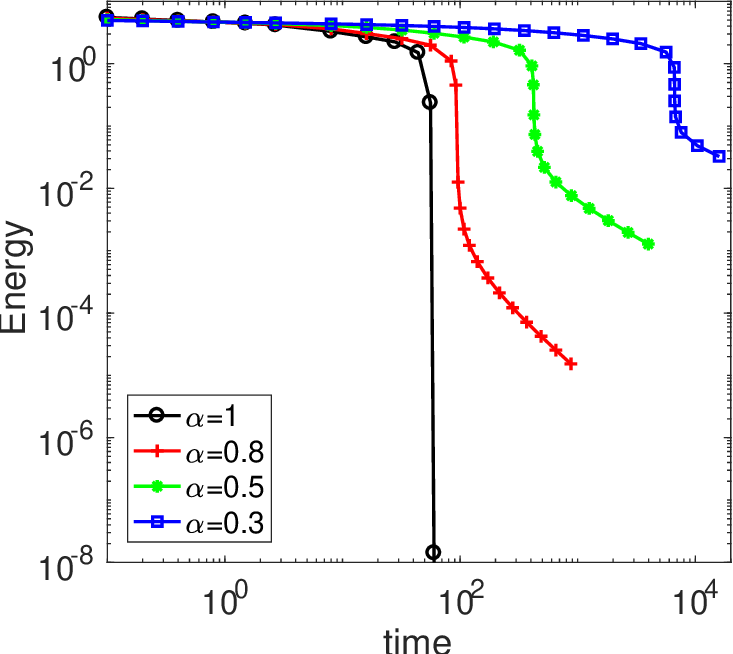}
  \caption{The energy dissipation for the time-fractional AC
    equation with different values of $\alpha$. }
  \label{fig:2FACEediss}
\end{figure}

In this case, we take $L_x=L_y=2, \veps=0.05,$ and $\gamma=0.05$. The
stabilization constant $S$ in scheme \eqref{eq:ACFDstab} is chosen as
$S=0.1$. We use $128\times 128$ Fourier modes in the physical domain,
and set the time step size as $\tau = 0.1$. The initial state is chosen
as
\begin{equation}
  \label{eq:iv_flower}
  \phi_0(x) =
  \tanh\left(\frac{1}{\sqrt{2}\veps} \Bigl( r
    -\frac14-\frac{1+\cos(4\theta)}{16}\Bigr)\right),
  \quad r = \sqrt{x^2+y^2},
  \; \theta = \arctan \frac{y}{x}.
\end{equation}
The phase field function and the energy evolution with various
fractional orders $\alpha$ are investigated. Fig. \ref{fig:1FACEsolu}
presents the phase field function $\phi$ with $\alpha = 0.3, 0.5, 1$ at
different time levels. It is clearly seen that it takes longer time to
reach the equilibrium if $\alpha$ becomes smaller. This observation can
be further verified by looking at the energy dissipation curves
displayed in Fig.~\ref{fig:2FACEediss}. It is seen from this figure that
the energy dissipation for the time-fractional AC equation has a long
tail effect.

\subsection{Numerical results for the time-fractional CH
  equation}

\begin{figure}[htbp!]
  \centering
  \includegraphics[width=0.28\textwidth]{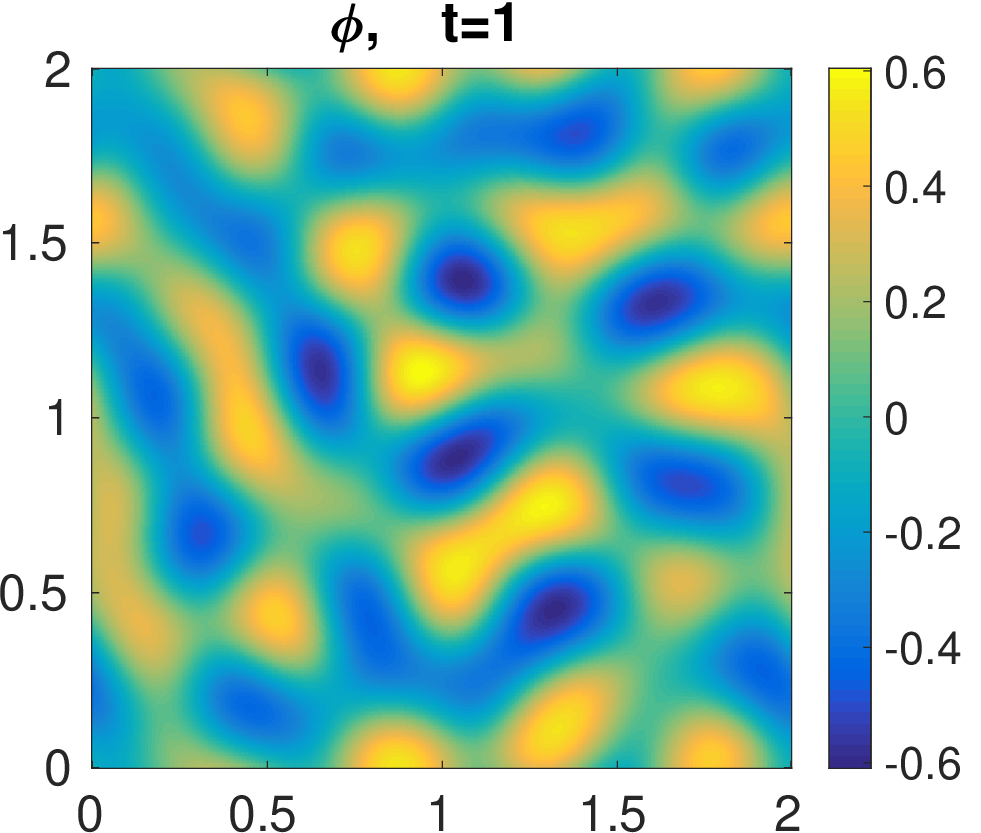}
  \includegraphics[width=0.28\textwidth]{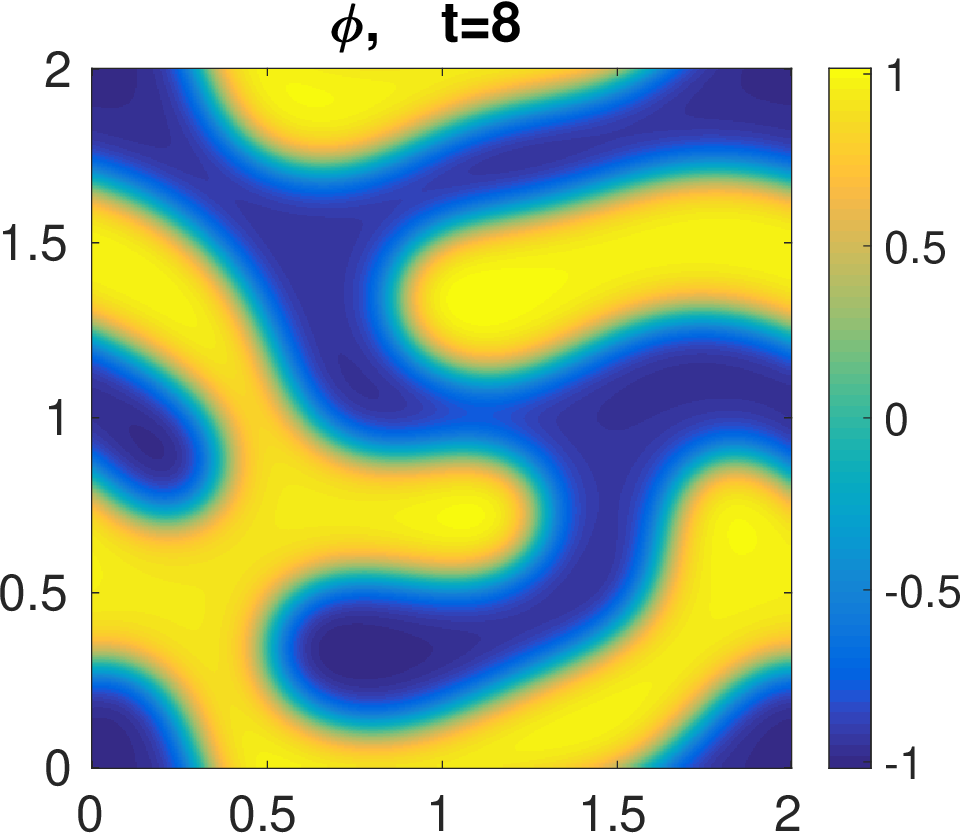}
  \includegraphics[width=0.28\textwidth]{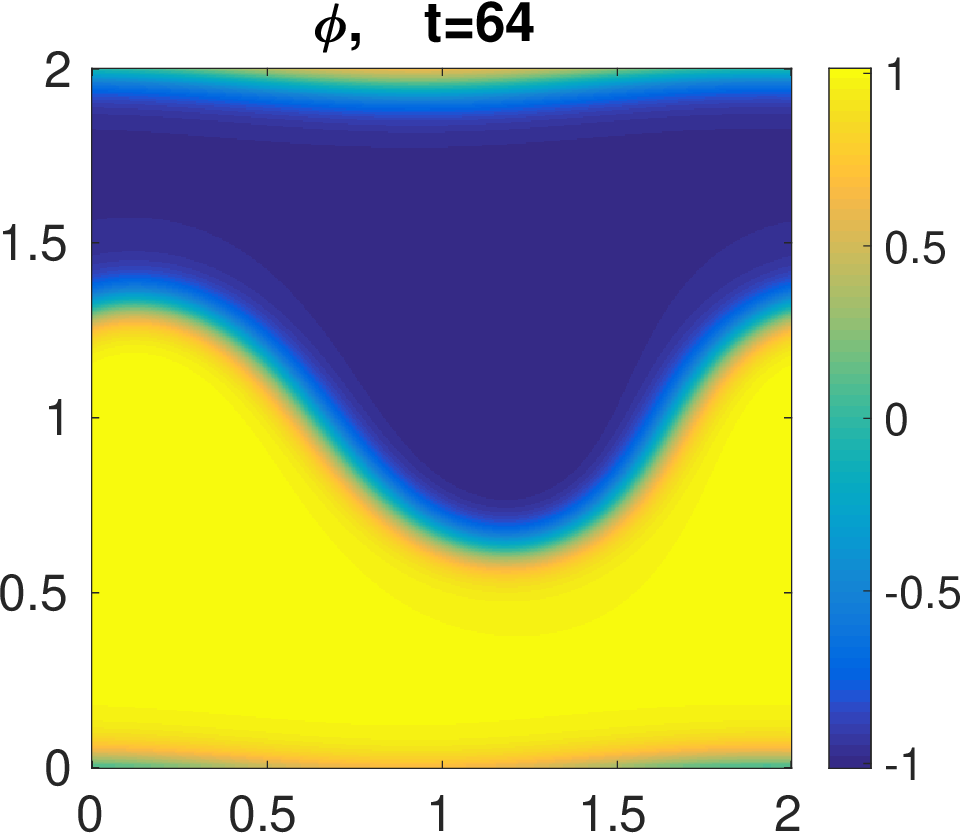}\\
  \includegraphics[width=0.28\textwidth]{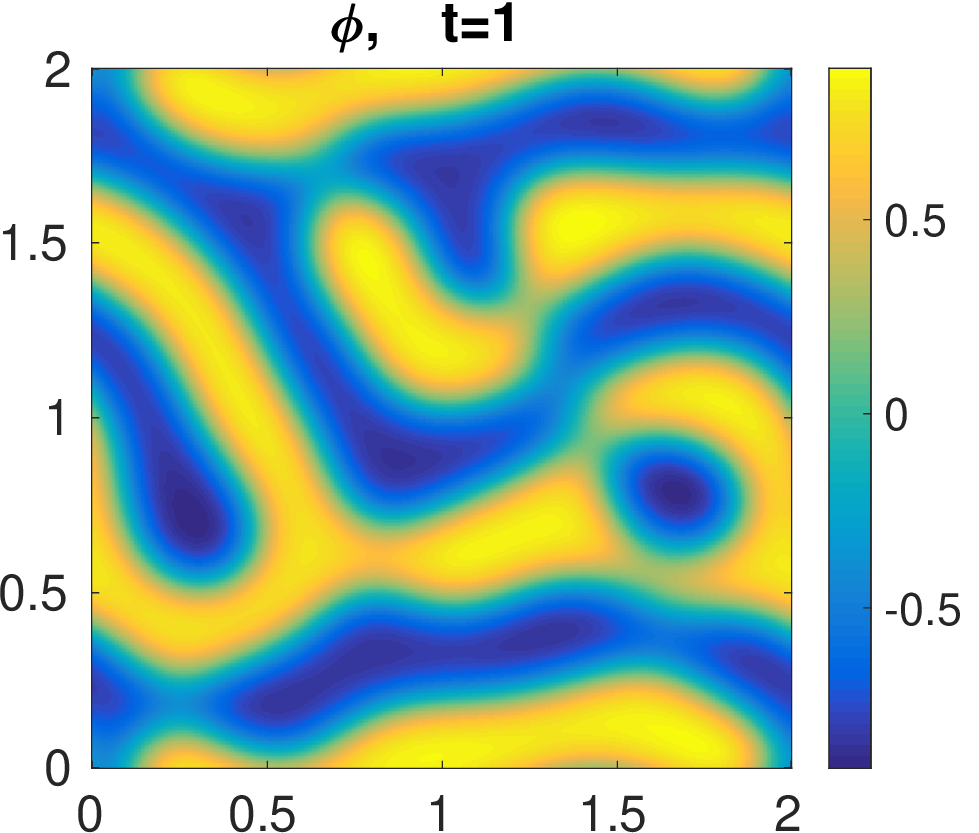}
  \includegraphics[width=0.28\textwidth]{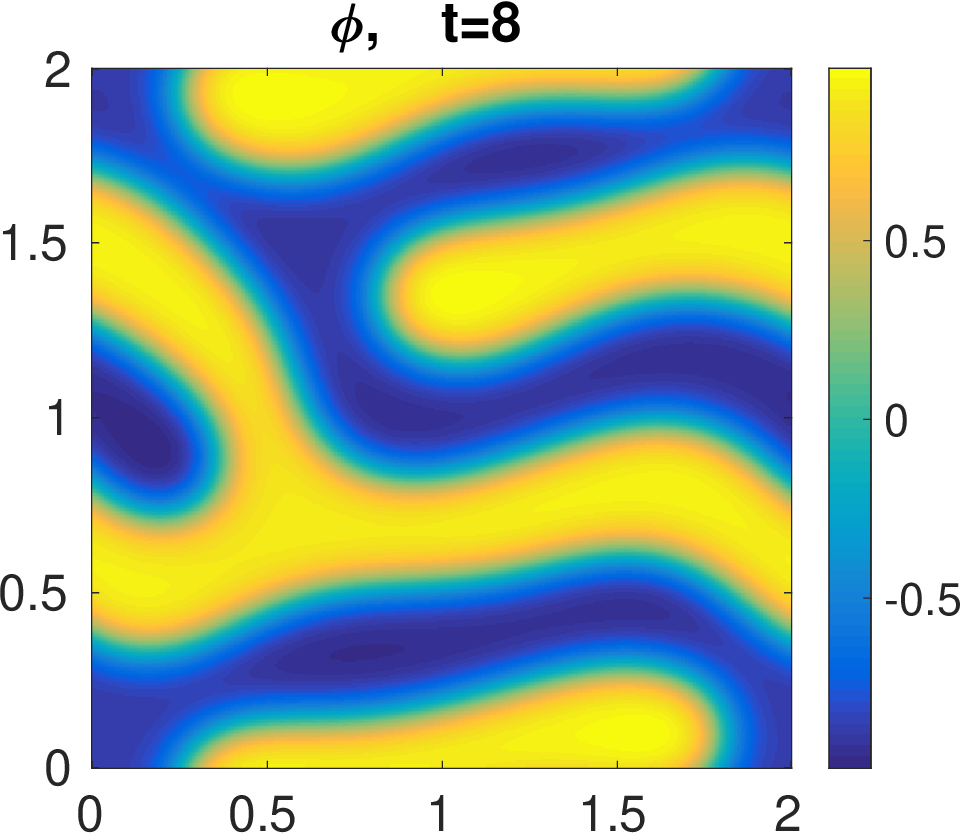}
  \includegraphics[width=0.28\textwidth]{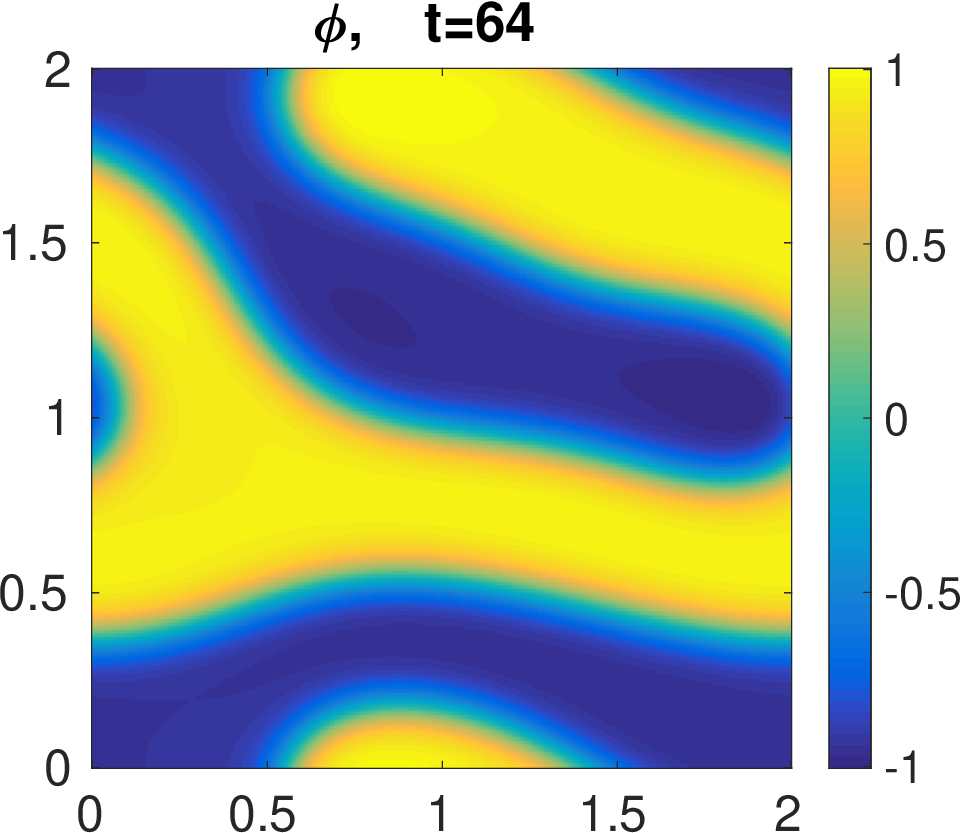}\\
  \includegraphics[width=0.28\textwidth]{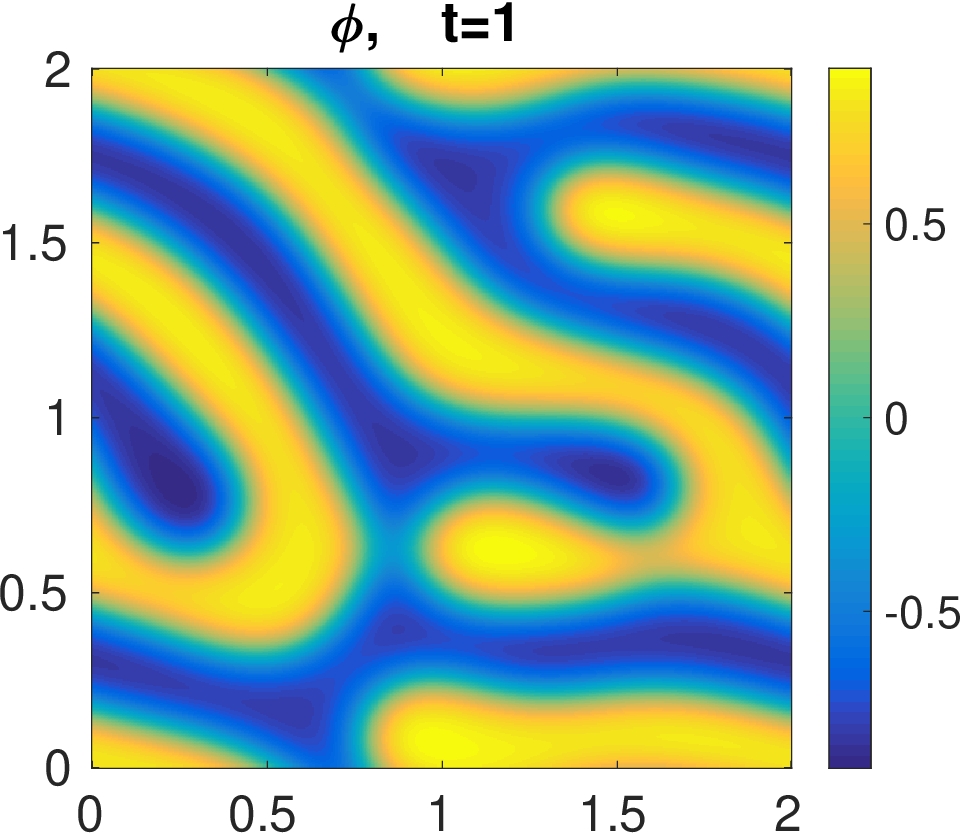}
  \includegraphics[width=0.28\textwidth]{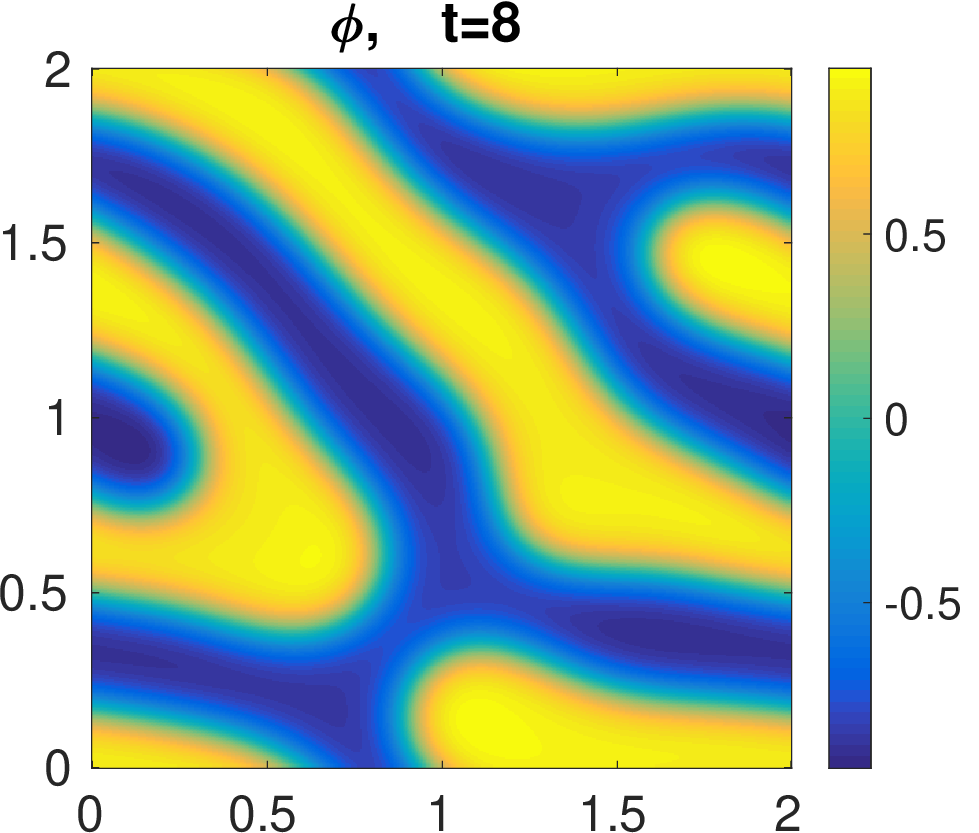}
  \includegraphics[width=0.28\textwidth]{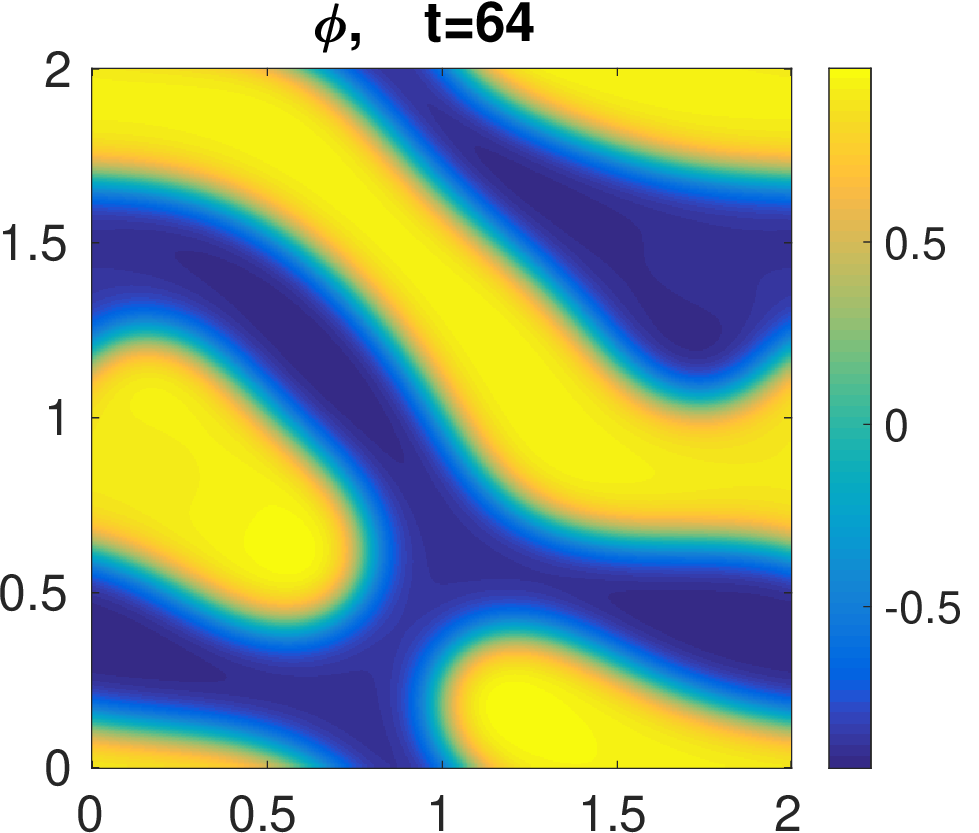}
  \caption{The snapshots of the solution of the time-fractional CH
    equation with $\alpha=1, 0.5, 0.3 $ (top, middle and bottom row,
    respectively). }
  \label{fig:3FCHEsolu}
\end{figure}

\begin{figure}[!htb]
  \centering
  \subfigure[The energy dissipation]{
    \includegraphics[width=0.45\textwidth]{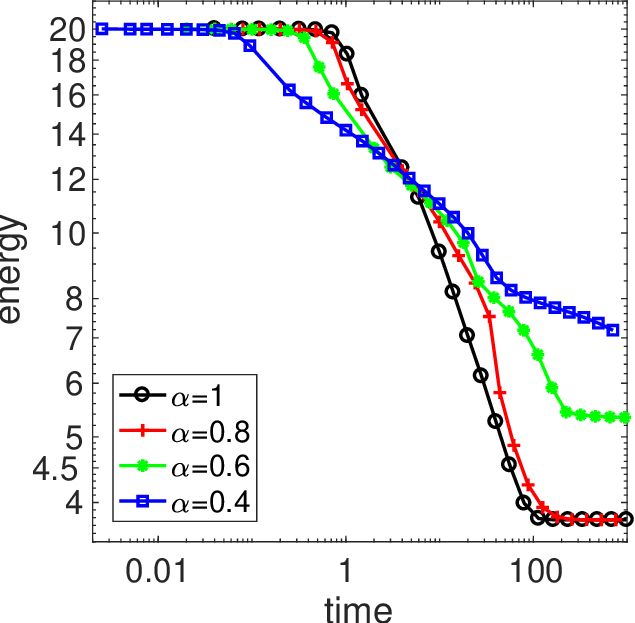}
    \label{fig:4FCHEediss}
  }
  \subfigure[The power law]{
    \includegraphics[width=0.45\textwidth]{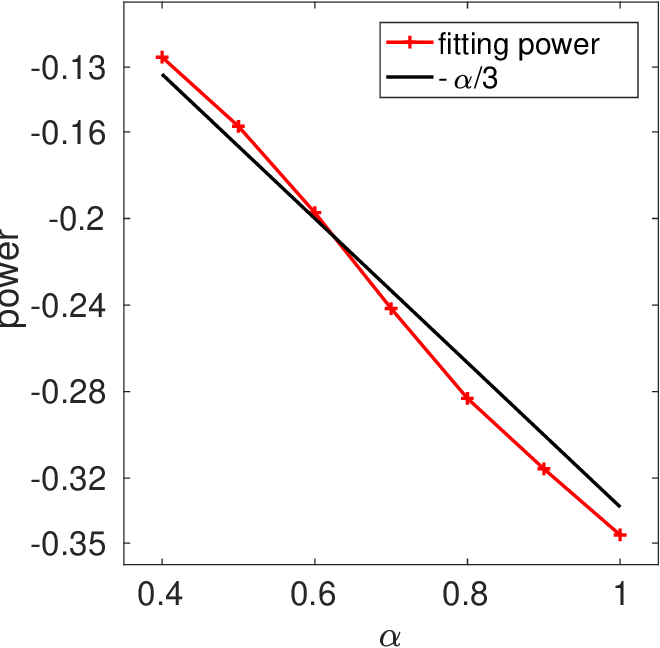}
    \label{fig:4FCHEpowerlaw}
  }
  \caption{The energy dissipation and the power-law for the
    time-fractional CH equation with different values of $\alpha$. }
  \label{fig:4FCHEepw}
\end{figure}

In this case, we take $L_x=L_y=2, \veps=0.05,$ and $\gamma=\veps^2$. A
uniformly random distribution field in $[-1,1]$ is taken as the initial
state, and the stabilization constant in scheme
\eqref{eq:TFCHEBDF1}-\eqref{eq:TFCHEBDF1_2} is chosen as $S=0.01$.  We
use $256\times256$ Fourier modes for spatial discretization. The time
step size is taken as $\tau = 0.001$. As shown in Section 3, the
numerical scheme is unconditionally stable if a suitable stabilization
constant is used.

We investigate again the time variation for the phase field function and
the energy evolution with different values of fractional parameters. Fig
\ref{fig:3FCHEsolu} presents the phase field function $\phi$ with
$\alpha = 0.3, 0.5, 1$ at different time levels. Again it is observed
that as $\alpha$ decreases the relaxation time reaching the equilibrium
{increases}. This assertion can be further verified by checking the
energy dissipation curves presented in Fig. \ref{fig:4FCHEediss}.  The
overall energy dissipation process can be split into three stages. In
the first stage, the bulk force is the driving force, and consequently
small scale phase separations form. In the second stage, small
structures interact with each other, so that an energy dissipation
power-law can be observed. In particular, the power-law behaves like
$E[\phi(t)] \approx C_\alpha t ^{p_\alpha}$. It is found by data fitting
that $p_\alpha \approx -\alpha/3$ (cf.  Fig. \ref{fig:4FCHEpowerlaw}),
which is consistent with the $-1/3$-law for the classical CH equation
(i.e., $\alpha =1$). In the last stage, the equilibrium solution with
the minimum energy is obtained.

\subsection{Numerical results for the time-fractional MBE model}

\begin{figure}[!htbp]
  \centering
  \includegraphics[width=0.3\textwidth]{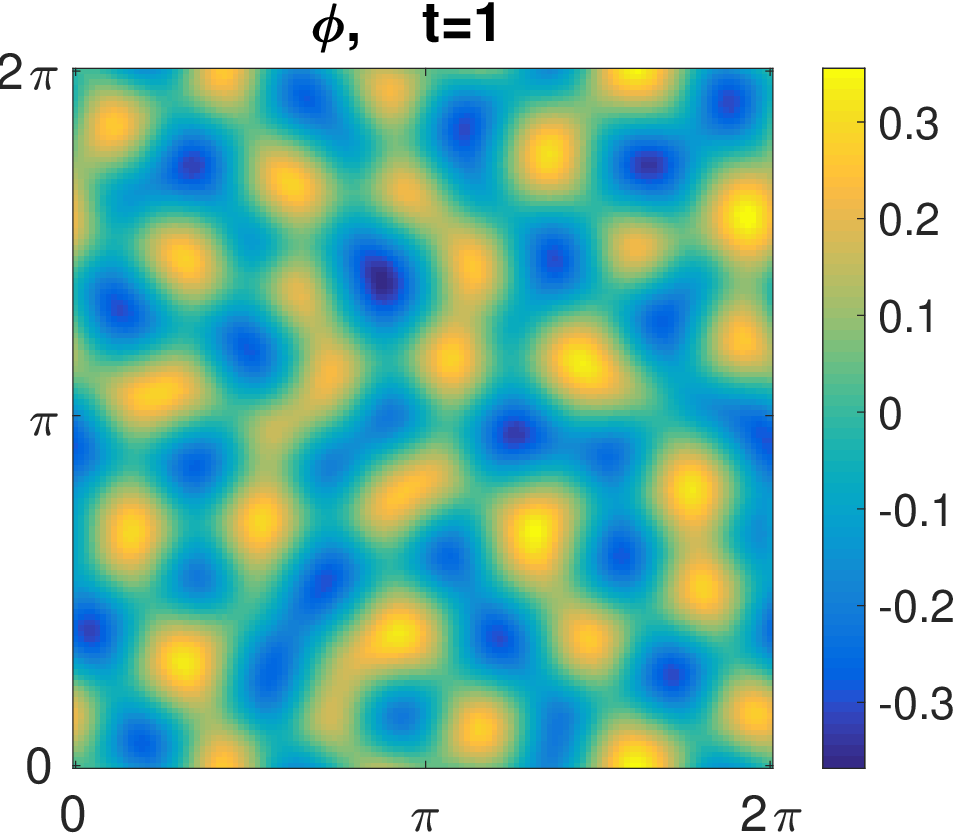}
  \includegraphics[width=0.3\textwidth]{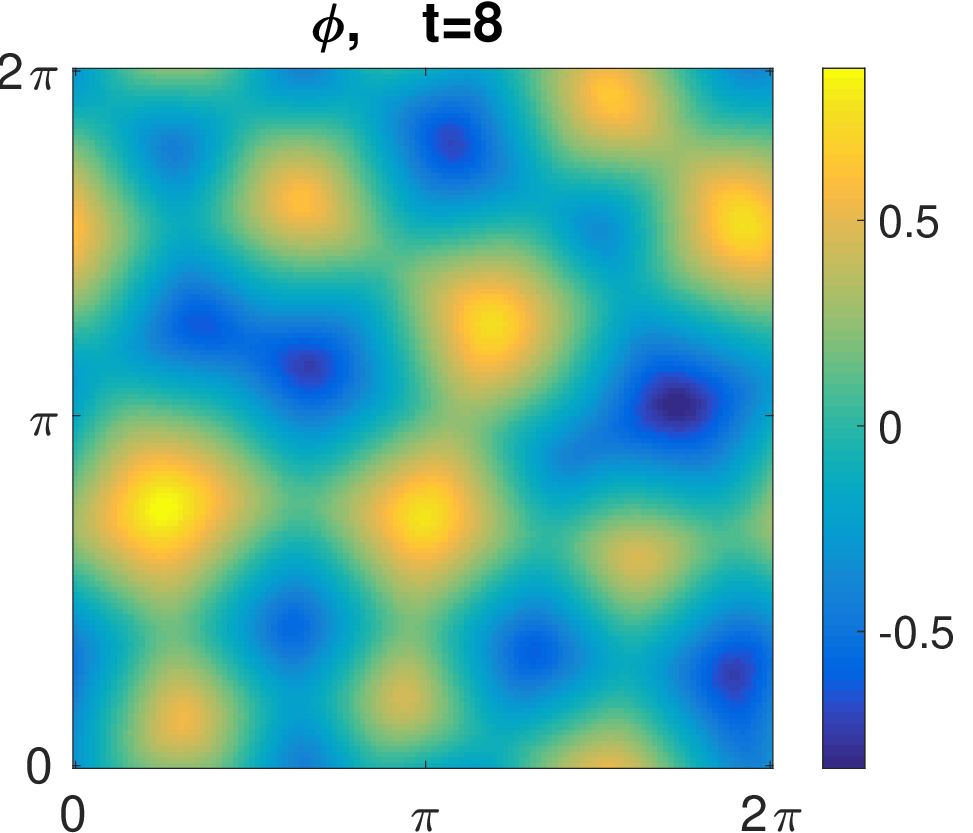}
  \includegraphics[width=0.3\textwidth]{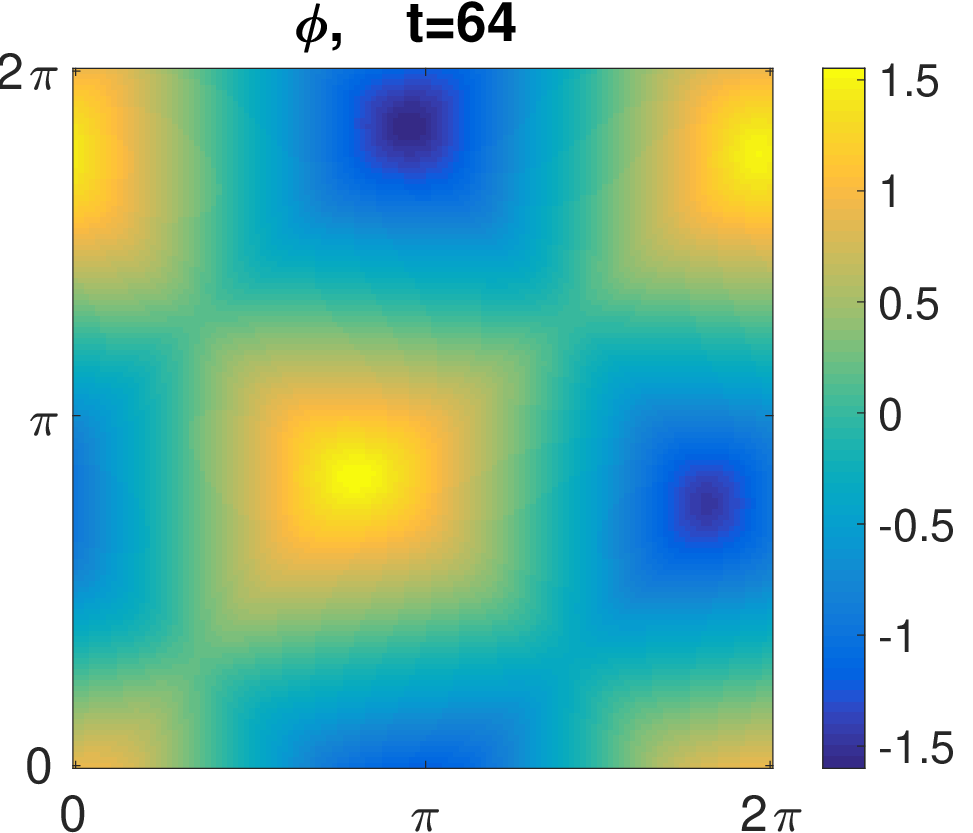}\\
  \includegraphics[width=0.3\textwidth]{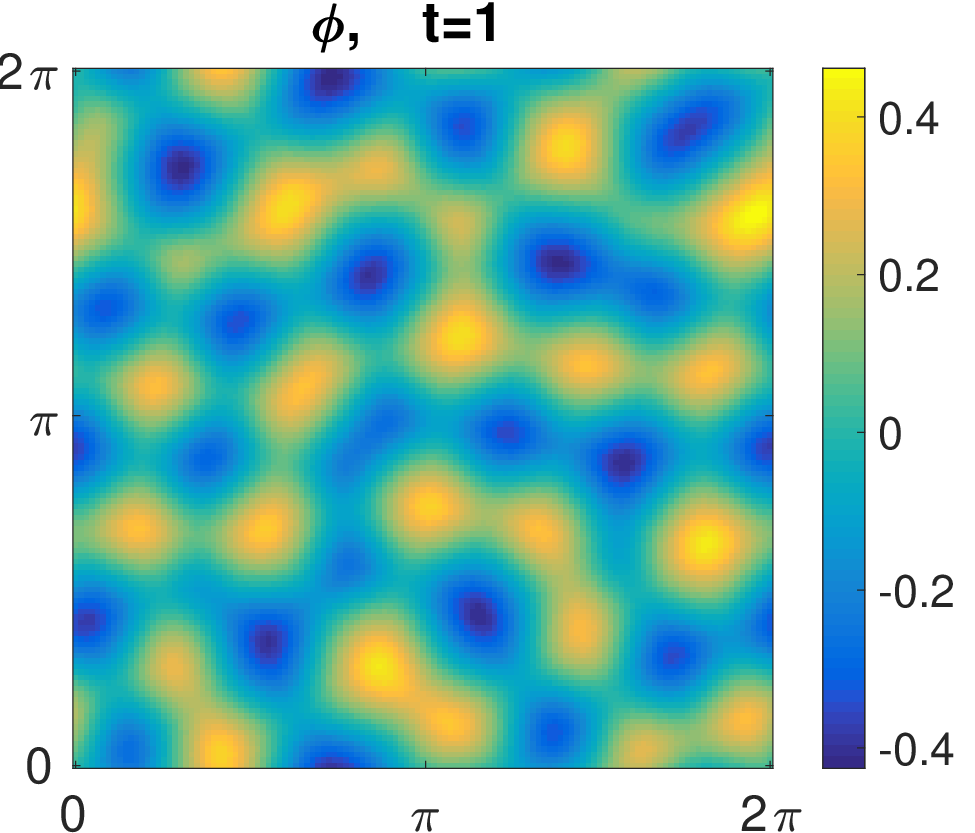}
  \includegraphics[width=0.3\textwidth]{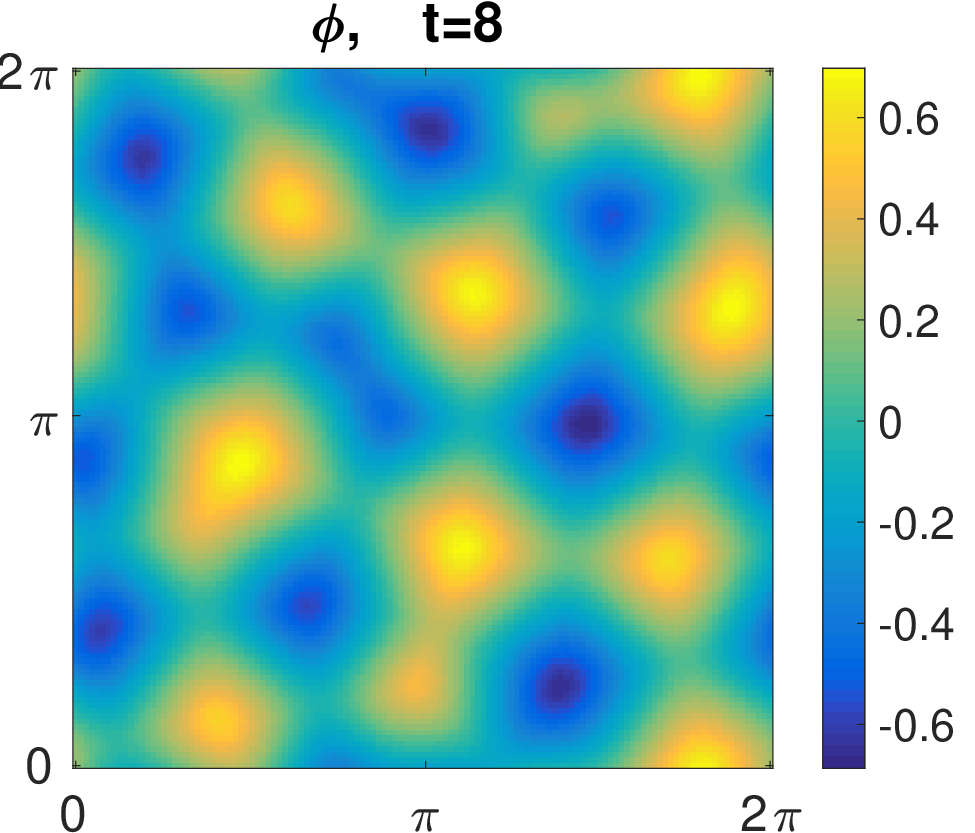}
  \includegraphics[width=0.3\textwidth]{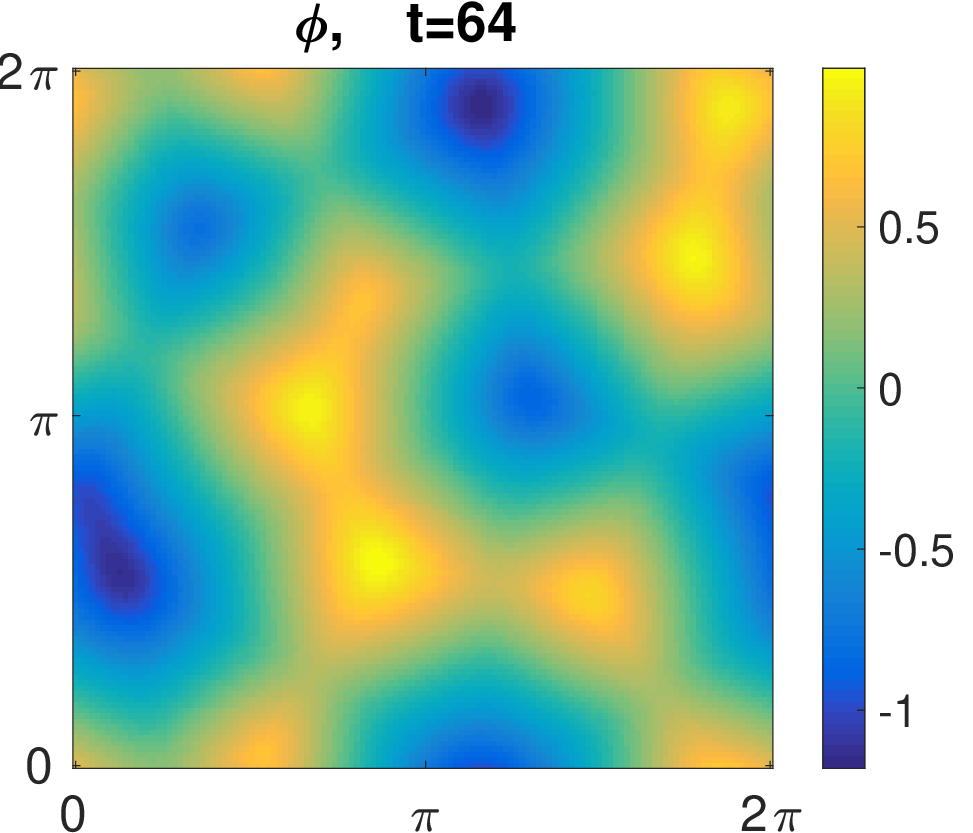}\\
  \includegraphics[width=0.3\textwidth]{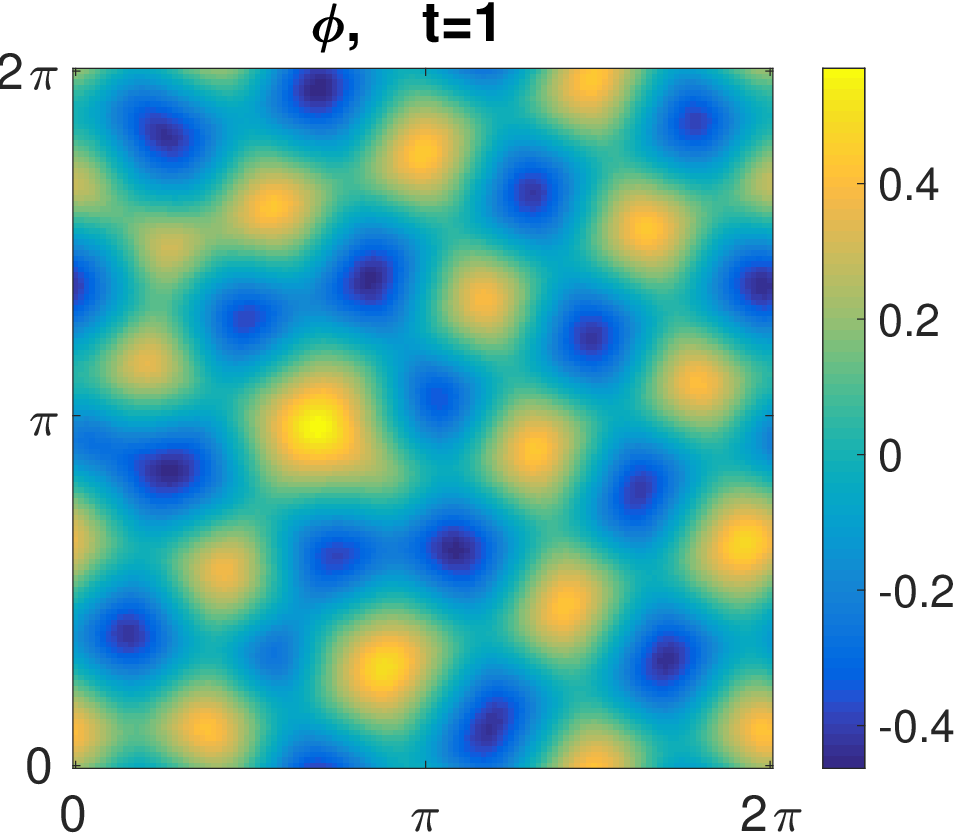}
  \includegraphics[width=0.3\textwidth]{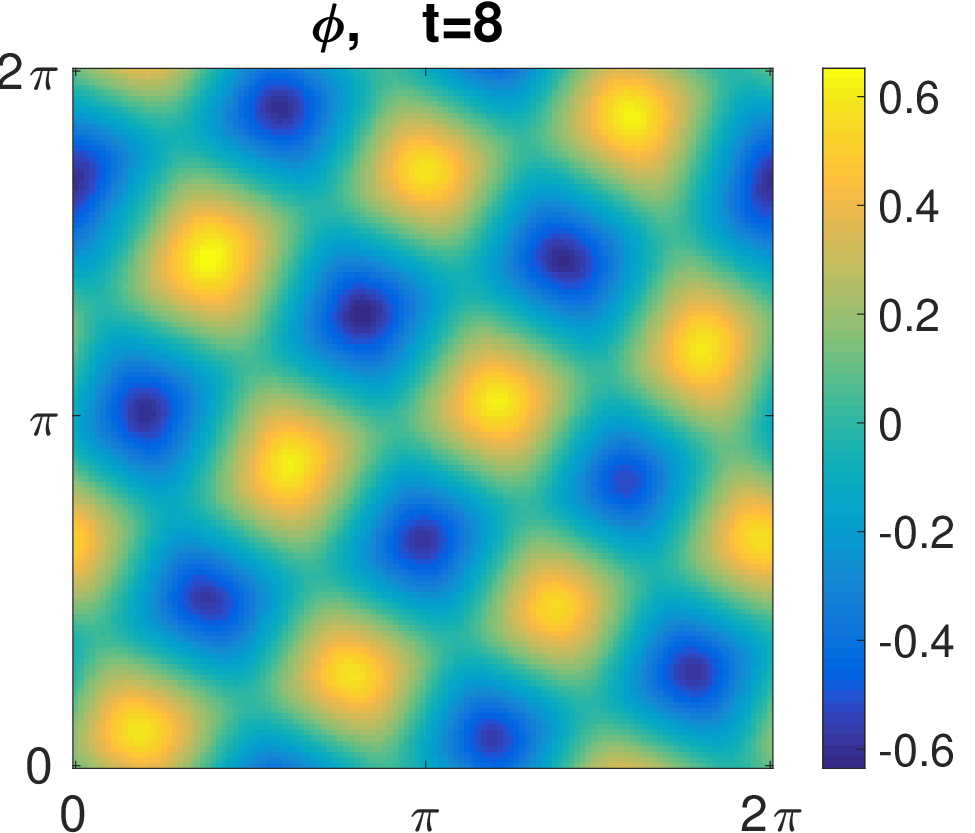}
  \includegraphics[width=0.3\textwidth]{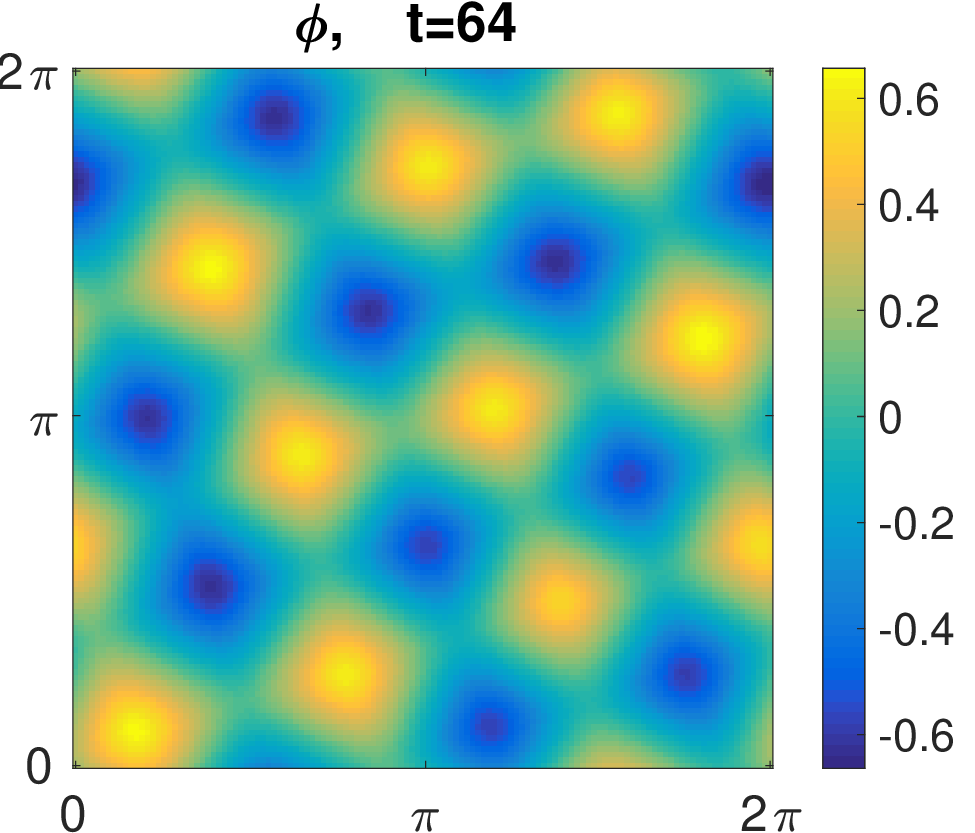}
  \caption{The snapshots of the numerical solution of the
    time-fractional MBE model with slop selection for
    $\alpha=1, 0.7$ and $0.4$ (top, middle and
    bottom row, respectively). }
  \label{fig:5FMBEsolu}
\end{figure}

\begin{figure}[!htb]
  \centering
  \subfigure[The energy dissipation]{
    \label{fig:6FMBEediss}
    \includegraphics[width=0.45\textwidth]{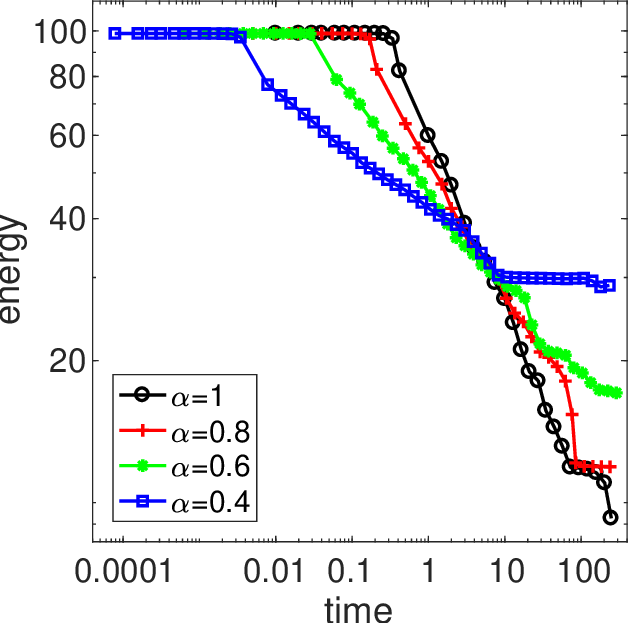}
  }
  \subfigure[The power-law scaling]{
    \label{fig:6FMBEpowerlaw}
    \includegraphics[width=0.45\textwidth]{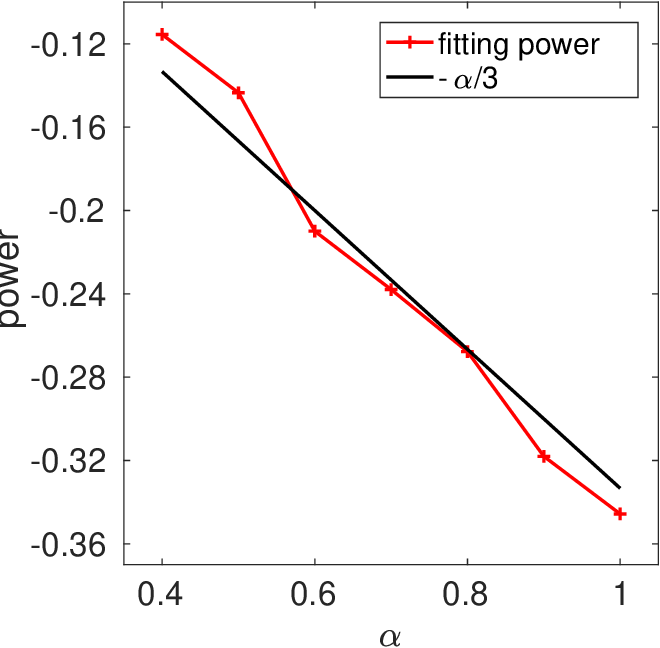}
  }
  \caption{The energy dissipation and power-law scaling for
    time-fractional MBE model with slop selection at several values of
    $\alpha$. }
\end{figure}

For the time-fractional MBE model with slope selection, we take
$L_x=L_y=2\pi, \veps=0.1,$ and $\gamma=\veps$. A uniformly random
distribution field {in $[-0.001,0.001]$} is chosen as the initial
state. We use the stabilized scheme \eqref{eq:TFMBEBDF1} with $S=0.1$,
and take $256\times256$ Fourier modes for spatial discretization. The
time step size is taken as $\tau = 0.001$.

The time evolution for the phase field function $\phi$ with
$\alpha = 0.4, 0.7, 1$ is presented in Fig \ref{fig:5FMBEsolu}, which
demonstrates again that as $\alpha$ decreases the relaxation time
reaching the equilibrium increases. It is seen from
Fig. \ref{fig:6FMBEediss} that the overall energy dissipation process
for the time-fractional MBE model consists of three stages, and in the
second stage a power-law with an asymptotic power of $-\alpha/3$ is
observed, which is confirmed numerically by
Fig. \ref{fig:6FMBEpowerlaw}.

\smallskip

\section{Concluding remarks}

In this work, we established an energy dissipation theory for the
time-fractional phase field equations. We prove in the continuous level
that the time-fractional phase field equations admit an energy
dissipation law of integral type.  In the discrete level, we propose a
class of finite difference schemes that can inherit the discrete energy
dissipation property. These numerical schemes are applied and analyzed
for the time-fractional AC equation, the time-fractional CH equation,
and the time-fractional MBE model. Several numerical experiments are
carried out to verify the theoretical predictions.  In particular, it is
observed numerically that the energy dissipation rate satisfies a power
law with an asymptotic power $-\alpha/3$, where $\alpha$ is the
fractional parameter.

We conclude this work by pointing out several relevant issues which
require further study.

\begin{itemize}
\item As discussed in Remark 2.1, we have presented the energy law
  $E[\phi(T)]\le E[\phi(0)]$. An open question is to verify the
  following energy dissipation law:
  \begin{align*}
    \frac{d}{dt} E\le 0  \quad  or  \quad  \frac{d^\alpha}{dt^\alpha} E\le 0.
  \end{align*}

\item We have shown numerically that the energy dissipation rate
  satisfies a power law with an asymptotic power $-\alpha/3$ in the
  coarsening stage for the time-fractional CH equation and the MBE
  model. However, a rigorous theoretical justification is needed.

\item On the numerical side, only first order schemes are investigated
  in this work.  It will be more useful and more challenging to design
  high-order energy stable schemes in our future studies.

\end{itemize}

\section*{Acknowledgment}

We thank the anonymous referees for their valuable comments and
suggestions which helped us to improve the manuscript a lot. We also
would like to thank Prof. Jie Shen, Prof. Yifa Tang, Dr. Jiwei Zhang,
Dr. Zhi Zhou and Dr. Honglin Liao for helpful discussions.  The
computations were partly done on the high performance computers of State
Key Laboratory of Scientific and Engineering Computing, Chinese Academy
of Sciences.

\bibliographystyle{plain}
\bibliography{fractional}

\begin{thebibliography}{10}

\bibitem{fractional_Mao}
M.~Ainsworth and Z.~Mao.
\newblock Analysis and approximation of a fractional {C}ahn-{H}illiard
  equation.
\newblock {\em SIAM J. Numer. Anal.}, 55(4):1689--1718, 2017.

\bibitem{fractional_III}
M.~Ainsworth and Z.~Mao.
\newblock Well-posedness of the {C}ahn-{H}illiard equation with fractional free
  energy and its {F}ourier {G}alerkin approximation.
\newblock {\em Chaos Soliton. Fract.}, 102:264--273, 2017.

\bibitem{fractional_I}
G.~Akagi, G.~Schimperna, and A.~Segatti.
\newblock Fractional {C}ahn-{H}illiard, {A}llen-{C}ahn and porous medium
  equations.
\newblock {\em J. Differ. Equations}, 261:2953--2985, 2016.

\bibitem{Allen.etal2016}
M.~Allen, L.~Caffarelli, and A.~Vasseur.
\newblock A parabolic problem with a fractional time derivative.
\newblock {\em Arch Rational Mech Anal}, 221(2):603--630, 2016.

\bibitem{AllenCahn_1979}
S.M. Allen and J.W. Cahn.
\newblock A microscopic theory for antiphase boundary motion and its
  application to antiphase domain coarsening.
\newblock {\em Acta Metall}, 27:1085--1095, 1979.

\bibitem{BoschKS_CiCP2018}
J.~Bosch, C.~Kahle, and M.~Stoll.
\newblock Preconditioning of a coupled {C}ahn-{H}illiard {N}avier-{S}tokes
  system.
\newblock {\em Commun. Comput. Phys}, 23:603--628, 2018.

\bibitem{Brunner2004}
Hermann Brunner.
\newblock {\em Collocation Methods for {{Volterra}} Integral and Related
  Functional Differential Equations}, volume~15.
\newblock {Cambridge University Press}, 2004.

\bibitem{caffarelli_l_1995}
L.~A. Caffarelli and N.~E. Muler.
\newblock An ${L^\infty}$ bound for solutions of the {Cahn}-{Hilliard}
  equation.
\newblock {\em Arch. Rational Mech. Anal.}, 133(2):129--144, December 1995.

\bibitem{CH}
J.~W. Cahn and J.~E. Hilliard.
\newblock Free energy of a nonuniform system {I}: {I}nterfacial free energy.
\newblock {\em J. Chem. Phys}, 28(2):258--267, 1958.

\bibitem{MBE}
S.~Clarke and D.~D. Vvedensky.
\newblock Origin of reflection high-energy electron-diffraction inten- sity
  oscilations during molecular-beam epitaxy: A computational modeling approach.
\newblock {\em Phys.Rev. Lett.}, 58:2235--2238, 1987.

\bibitem{condette_spectral_2011}
N.~Condette, C.~Melcher, and E.~{S\"uli}.
\newblock Spectral approximation of pattern-forming nonlinear evolution
  equations with double-well potentials of quadratic growth.
\newblock {\em Math. Comp.}, 80(273):205--223, 2011.

\bibitem{Diegel.etal2017}
A.~E. Diegel, C.~Wang, X.~Wang, and S~. Wise.
\newblock Convergence analysis and error estimates for a second order accurate
  finite element method for the {{Cahn}}-{{Hilliard}}-{{Navier}}-{{Stokes}}
  system.
\newblock {\em Numer. Math.}, 137(3):495--534, 2017.

\bibitem{du_phase_2004}
Q.~Du, C.~Liu, and X.~Wang.
\newblock A phase field approach in the numerical study of the elastic bending
  energy for vesicle membranes.
\newblock {\em J. Comput. Phys.}, 198(2):450--468, 2004.

\bibitem{du_numerical_1991}
Q.~Du and Roy~A. Nicolaides.
\newblock Numerical analysis of a continuum model of phase transition.
\newblock {\em SIAM J. Numer. Anal.}, 28(5):1310--1322, 1991.

\bibitem{Du}
Q.~Du, J.~Yang, and Z.~Zhou.
\newblock Time-fractional {A}llen-{C}ahn equations: Analysis and numerical
  methods.
\newblock {\em arXiv:1906.06584}, 2019.

\bibitem{elliott_global_1993}
C.~M. Elliott and A.~M. Stuart.
\newblock The global dynamics of discrete semilinear parabolic equations.
\newblock {\em SIAM J. Numer. Anal.}, 30:1622--1663, 1993.

\bibitem{Eyre_1998}
D.~J. Eyre.
\newblock Unconditionally gradient stable time marching the {Cahn}-{Hilliard}
  equation.
\newblock In {\em Computational and Mathematical Models of Microstructural
  Evolution ({San} {Francisco}, {CA}, 1998)}, volume 529 of {\em Mater. {Res}.
  {Soc}. {Sympos}. {Proc}.}, pages 39--46. MRS, 1998.

\bibitem{Gomez_JCP_2011}
H.~Gomez and T.~Hughes.
\newblock Provably unconditionally stable, second-order time-accurate, mixed
  variational methods for phase-field models.
\newblock {\em J. Comput. Phys.}, 230:5310--5327, 2011.

\bibitem{Guo.etal2016}
J.~Guo, C.~Wang, S.~Wise, and X.~Yue.
\newblock An ${H}^2$ convergence of a second-order convex-splitting, finite
  difference scheme for the three-dimensional {{Cahn}}-{{Hilliard}} equation.
\newblock {\em Commun. Math. Sci.}, 14(2):489--515, 2016.

\bibitem{Hawkins_tumor_2012}
A.~Hawkins-Daarud, K.~G. Van~Der Zee, and J.~T. Oden.
\newblock Numerical simulation of a thermodynamically consistent four-species
  tumor growth model.
\newblock {\em Int. J. Numer. Methods Biomed Eng.}, 8:3--24, 2012.

\bibitem{jiang_fast_2017}
S.~Jiang, J.~Zhang, Q.~Zhang, and Z.~Zhang.
\newblock Fast evaluation of the {C}aputo fractional derivative and its
  applications to fractional diffusion equations.
\newblock {\em Commun. Comput. Phys.}, 21(3):650--678, 2017.

\bibitem{Jin.etal2013}
B.~Jin, R.~Lazarov, and Z.~Zhou.
\newblock Error estimates for a semidiscrete finite element method for
  fractional order parabolic equations.
\newblock {\em SIAM J. Numer. Anal.}, 51(1):445--466, 2013.

\bibitem{jin_numerical_2018}
B.~Jin, B.~Li, and Z.~Zhou.
\newblock Numerical analysis of nonlinear subdiffusion equations.
\newblock {\em SIAM J. Numer. Anal.}, 56(1):1--23, 2018.

\bibitem{kilbas_theory_2006}
A.~Kilbas, H.~Srivastava, and J.~Trujillo.
\newblock {\em Theory and Applications of Fractional Differential Equations}.
\newblock North-Holland Mathematics Studies, 2006.

\bibitem{Le.etal2016}
K.~Le, W.~McLean, and K.~Mustapha.
\newblock Numerical solution of the time-fractional {{Fokker}}-{{Planck}}
  equation with general forcing.
\newblock {\em SIAM J. Numer. Anal.}, 54(3):1763--1784, 2016.

\bibitem{Li_Liu_2003}
B.~Li and J.-G. Liu.
\newblock Thin film epitaxy with or without slope selection.
\newblock {\em Eur. J. Appl. Math.}, 14:713--743, 2003.

\bibitem{Li_tumor_2007}
X.~Li, V.~Cristini, Q.~Nie, and J.S. Lowengrub.
\newblock Nonlinear three-dimensional simulation of solid tumor growth.
\newblock {\em Discrete Contin. Dyn. Syst., Ser. B}, 7:581--604, 2007.

\bibitem{fractional_Wang_I}
Z.~Li, H.~Wang, and D.~Yang.
\newblock A space-time fractional phase-field model with tunable sharpness and
  decay behavior and its efficient numerical simulation.
\newblock {\em J. Comput. Phys.}, 347:20--38, 2017.

\bibitem{liao_discrete_2019}
H.~Liao, W.~McLean, and J.~Zhang.
\newblock A discrete {Gr\"onwall} inequality with applications to numerical
  schemes for subdiffusion problems.
\newblock {\em SIAM J. Numer. Anal.}, 57(1):218--237, 2019.

\bibitem{lin_finite_2007}
Y.~Lin and C.~Xu.
\newblock Finite difference/spectral approximations for the time-fractional
  diffusion equation.
\newblock {\em J. Comput. Phys.}, 225(2):1533--1552, 2007.

\bibitem{LowengrubT_1998}
J.~Lowengrub and L.~Truskinovsky.
\newblock Quasi-incompressible {C}ahn-{H}illiard fluids and topological
  transitions.
\newblock {\em P. Roy. Soc. Lond. A. Math. Phy.}, 454:2617--2654, 1998.

\bibitem{Ma_etal_CiCP2017}
L.~Ma, R.~Chen, X.~Yang, and H.~Zhang.
\newblock Numerical approximations for {A}llen-{C}ahn type phase field model of
  two-phase incompressible fluids with moving contact lines.
\newblock {\em Commun. Comput. Phys.}, 27:867--889, 2017.

\bibitem{Anderson_1998}
D.~M. Anderson G.~B. McFadden and A.~A. Wheeler.
\newblock Diffuse-interface methods in fluid mechanics.
\newblock {\em Annu. Rev. Fluid Mech.}, 30:139--165, 1998.

\bibitem{mclean_convergence_2009}
W.~McLean and K.~Mustapha.
\newblock Convergence analysis of a discontinuous {Galerkin} method for a
  sub-diffusion equation.
\newblock {\em Numer. Algor.}, 52(1):69--88, 2009.

\bibitem{Mustapha.Schotzau2014}
K.~Mustapha and D.~Sch\"otzau.
\newblock Well-posedness of hp-version discontinuous {{Galerkin}} methods for
  fractional diffusion wave equations.
\newblock {\em IMA J. Numer. Anal.}, 34(4):1426--1446, 2014.

\bibitem{fractional_II}
Y.~Nec, A.A. Nepomnyashchy, and A.A. Golovin.
\newblock Front-type solutions of fractional {A}llen-{C}ahn equation.
\newblock {\em Physica D}, 237:3237--3251, 2008.

\bibitem{Nohel.Shea1976}
J.~A Nohel and D.~F Shea.
\newblock Frequency domain methods for {{Volterra}} equations.
\newblock {\em Adv. Math.}, 22(3):278--304, 1976.

\bibitem{podlubny_fractional_1998}
Igor Podlubny.
\newblock {\em Fractional Differential Equations}.
\newblock Elsevier, 1998.

\bibitem{qian_molecular_2003}
T.~Qian, X.~P. Wang, and P.~Sheng.
\newblock Molecular scale contact line hydrodynamics of immiscible flows.
\newblock {\em Phy. Rev. E}, 68(1):016306, 2003.

\bibitem{Shen.etal2016c}
J.~Shen, T.~Tang, and J.~Yang.
\newblock On the maximum principle preserving schemes for the generalized
  {{Allen}}-{{Cahn}} equation.
\newblock {\em Commun. Math. Sci.}, 14(6):1517--1534, 2016.

\bibitem{ShenY_DCDS_2010}
J.~Shen and X.~Yang.
\newblock Numerical approximations of {A}llen-{C}ahn and {C}ahn-{H}illiard
  equations.
\newblock {\em Discret. Contin. Dyn. Syst.}, 28:1669--1691, 2010.

\bibitem{shen_efficient_2015}
J.~Shen, X.~Yang, and H.~Yu.
\newblock Efficient energy stable numerical schemes for a phase field moving
  contact line model.
\newblock {\em J. Comput. Phys.}, 284:617--630, 2015.

\bibitem{fractional_Xu}
F.~Song, C.~Xu, and G.~Em Karniadakis.
\newblock A fractional phase-field model for two-phase flows with tunable
  sharpness: Algorithms and simulations.
\newblock {\em Comput. Methods Appl. Mech. Engrg}, 305:376--404, 2016.

\bibitem{stynes_error_2017}
M.~Stynes, E.~O\'Riordan, and J.~Gracia.
\newblock Error analysis of a finite difference method on graded meshes for a
  time-fractional diffusion equation.
\newblock {\em SIAM J. Numer. Anal.}, 55(2):1057--1079, 2017.

\bibitem{sun_fully_2006}
Z.~Sun and X.~Wu.
\newblock A fully discrete difference scheme for a diffusion-wave system.
\newblock {\em Appl. Numer. Math.}, 56(2):193--209, 2006.

\bibitem{Tang1993b}
T.~Tang.
\newblock A finite difference scheme for partial integro-differential equations
  with a weakly singular kernel.
\newblock {\em Appl. Numer. Math.}, 11(4):309--319, 1993.

\bibitem{wang_energy_2018}
L.~Wang and H.~Yu.
\newblock Energy stable second order linear schemes for the {Allen-Cahn}
  phase-field equation.
\newblock {\em Commun. Math. Sci.}, 17(3):609--635, 2018.

\bibitem{wang_efficient_2018}
L.~Wang and H.~Yu.
\newblock On efficient second order stabilized semi-implicit schemes for the
  {Cahn-Hilliard} phase field equation.
\newblock {\em J. Sci. Comput.}, 77:1185--1209, 2018.

\bibitem{Wise_tumor_2008}
S.~M. Wise, J.~S. Lowengrub, H.~B. Frieboes, and V.~Cristini.
\newblock Three-dimensional multispecies nonlinear tumor growth {I}: {Model}
  and numerical method.
\newblock {\em J. Theor. Biol.}, 253(3):524--543, 2008.

\bibitem{XuT_SINUM_2006}
C.~Xu and T.~Tang.
\newblock Stability analysis of large time-stepping methods for epitaxial
  growth models.
\newblock {\em {SIAM} J. {N}umer. {A}nal.}, 44:1759--1779, 2006.

\bibitem{xu_sharp-interface_2018}
X.~Xu, Y.~Di, and H.~Yu.
\newblock Sharp-interface limits of a phase-field model with a generalized
  {Navier} slip boundary condition for moving contact lines.
\newblock {\em J. Fluid Mech.}, 849:805--833, 2018.

\bibitem{Yan_CiCP_2018}
Y.~Yan, W.~Chen, C.~Wang, and S.~M. Wise.
\newblock A second-order energy stable {BDF} numerical scheme for the
  {C}ahn-{H}illiard equation.
\newblock {\em Commun. Comput. Phys}, 23:572--602, 2018.

\bibitem{YuJZ_CiCP2010}
H.~Yu, G.~Ji, and P.~Zhang.
\newblock A nonhomogeneous kinetic model of liquid crystal polymers and its
  thermodynamic closure approximation.
\newblock {\em Commun. Comput. Phys.}, 7:383--402, 2010.

\bibitem{fractional_Wang_II}
J.~Zhao, L~Chen, and H.~Wang.
\newblock On power law scaling dynamics for time-fractional phase field models
  during coarsening.
\newblock {\em Commun. Nonlinear Sci.}, 70:257--270, 2019.

\end{thebibliography}

\end{document}